\documentclass{article}

\usepackage{arxiv}          

\usepackage[utf8]{inputenc} 
\usepackage[T1]{fontenc}    
\usepackage{url}            
\usepackage{booktabs}       
\usepackage{nicefrac}       
\usepackage{microtype}      
\usepackage{lineno}         

\usepackage{amsmath}
\usepackage{amsfonts,amsthm}       
\usepackage{paralist}
\usepackage{graphics}
\usepackage{epsfig}         
\usepackage{graphicx}  
\usepackage{epstopdf}       
\usepackage[colorlinks=true]{hyperref}
\hypersetup{urlcolor=blue, citecolor=red}

\usepackage{subcaption}     
\usepackage{algorithm}      
\usepackage{algorithmic}   
\usepackage{makecell}
\newtheorem{theorem}{Theorem}[section]
\newtheorem{corollary}{Corollary}

\newtheorem*{problem}{Open Problem}
\newtheorem{definition}[theorem]{Definition}
\newtheorem{remark}{Remark}

\newcommand{\R}{\mathbb{R}}

\newcommand{\bfh}{\mathbf{h}}
\newcommand{\neigh}{N^{\bar{C}}}
\newcommand{\change}{\Delta \chi^{\bar{C}}}

\newcommand{\myparagraph}[1]{\noindent {\bf #1} }
\newcommand{\ts}{\textsuperscript}

\def\foo substitute{terrain}
\title{Euler Characteristic Surfaces}

\author{Gabriele Beltramo\ts{1}\thanks{First author contacts. Email: g.beltramo@qmul.ac.uk; Personal webpage: https://gbeltramo.github.io/} \And Rayna Andreeva\ts{2} \And Ylenia Giarratano\ts{3} \And Miguel O. Bernabeu\ts{3}  \And Rik Sarkar\ts{2} \And Primoz Skraba\ts{1} \vspace*{4mm} \\ 
\ts{1} School of Mathematical Sciences, Queen Mary University of London, London, E1 4NS, UK \\
\texttt{g.beltramo@qmul.ac.uk,p.skraba@qmul.ac.uk} \\
\ts{2} School of Informatics, University of Edinburgh, Edinburgh, EH8 9AB, UK \\
\texttt{r.andreeva@ed.ac.uk, rsarkar@inf.ed.ac.uk} \\
\ts{3} Usher Institute, University of Edinburgh, Edinburgh, EH16 4UX, UK \\
\texttt{ylenia.giarratano@ed.ac.uk, miguel.bernabeu@ed.ac.uk}
}

\begin{document}
\maketitle

\begin{abstract}
We study the use of the Euler characteristic for multiparameter topological data analysis. Euler characteristic is a classical, well-understood topological invariant that has appeared in numerous applications, including in the context of random fields. The goal of this paper is to present the extension of using the Euler characteristic in higher-dimensional parameter spaces. While topological data analysis of higher-dimensional parameter spaces using stronger invariants such as homology continues to be the subject of intense research, Euler characteristic is more manageable theoretically and computationally, and this analysis can be seen as an important intermediary step in multi-parameter topological data analysis. We show the usefulness of the techniques using artificially generated examples, and a real-world application of detecting diabetic retinopathy in retinal images. 
\end{abstract}



\section{Introduction}
\label{sec:intro}
The field of topological data analysis (TDA) has attracted a lot of research over the last few years.  I this field, the most widely used tool is persistent homology, which provides a stable summary of a space/dataset over an entire range of parameter choices/scales. The application of this technique has been largely limited  to one dimensional parameter spaces due to both theoretical and computational challenges. 

In this paper, we present an alternative approach to studying data where higher-dimensional parameter spaces are naturally present.
Rather than using homology, we opt for a simpler topological invariant -- the Euler characteristic. We obtain highly efficient algorithms to compute summaries of datasets over multiple parameters. We refer to the summaries as \emph{Euler characteristic surfaces} or \emph{Euler surfaces} for short. The goal of this paper is to illustrate that Euler surfaces \textbf{can provide insight into the data over multidimensional parameter spaces}.

The Euler characteristic $\chi$  makes an appearance in many different fields of topology and geometry including: algebraic topology~\cite{hatcher2002algebraic}, differential geometry~\cite{guillemin2010differential}, and stochastic geometry and topology~\cite{adler2009random}. It has been generalized to highly abstract settings such as enriched categories~\cite{leinster2006euler}, and can be seen as generalized measure~\cite{baryshnikov2009target}. One of the truly remarkable aspects of the Euler characteristic is that it has allows for a local description which directly enables its efficient computation. In this paper, we provide a general algorithm for two common settings in TDA and a Python package. 


The idea of topological invariant of a sequence of spaces -- e.g. as found with persistent homology -- appears in the Euler characteristic domain as the \emph{Euler characteristic curve} (ECC), and has been used for \emph{topological inference} in a range of applicaitons~\cite{penny2011statistical}. We provide a brief overview of this work in the following section, with a focus on a data-driven approach, and efficacy in real datasets.


Our main contribution is to investigate the multi-parameter setting. Beginning with the work by Carlsson and Zomorodian \cite{carlsson2009theory}, there has been a large number of approaches proposed to deal with multi-parameter persistence. Approaches include directions such as the rank invariant~\cite{carlsson2009theory,cagliari2011finiteness,chacholski2015multidimensional}, microlocal analysis~\cite{kashiwara2018persistent}, and higher-dimensional analogues of persistence diagrams~\cite{mccleary2019multiparameter,harrington2019stratifying}. These all capture related but somewhat different concepts. To the best of our knowledge, the rank invariant is the only case where implementations exist, and although they perform well, the algorithms do not scale in the same way as one-dimensional persistence. While multidimensional persistent homology remains an active area of research, we note that there is substantial evidence that in many settings -- particularly in presence of  randomness -- there is a surprisingly little loss of information in going from homology to the Euler characteristic. It has been observed that in many random models, at any given parameter, the homology of a single dimension dominates~\cite{kahle2014topology,bobrowski2020homological}. Therefore, studying the Euler characteristic at different parameter values can provide a lot of topological information. 

As illustrated by the applications of the ECC, the Euler characteristic provides a useful functional summary of data, which can readily be used for classification -- especially as closed form expressions often exist. The extension to higher-dimensional parameter spaces clearly enables more discriminative summaries. More importantly, by considering the difference of Euler characteristics, \textbf{we can  identify interesting regions of parameter space}. 

We present several instances where the ``shape" of Euler surfaces and the difference of Euler surfaces provide interesting information about the underlying generating process and the parameter space. In addition to simulated data, we also present a real-world application:  detecting diabetic retinopathy (abbreviated DR). In diabetic patients, one of the early manifestations of this disease is change of the structure of blood vessels in the retina (See Figure~\ref{fig:comp}). Accurate detection of such changes may help in early detection of the disease and prevention of significant damage. Recent work in this area has been a series of machine learning approaches aimed at automated detection of DR and other diseases from retinal images~\cite{stolte2020survey}. These works have largely used methods such as neural networks, which are accurate only with large training data volumes and are not easily interpretable. In diagnostic medicine, datasets are often small and interpretability is paramount. On multiple datasets of retinal images, we illustrate two key points. First: the ECC is already effective at detecting DR, and secondly: expanding to the multiple parameters can yield insights into the qualitative differences between the blood vessels in healthy patients and those suffering from DR. 

The overall goal of this paper is to highlight this natural extension to existing work, as we believe this may lead to an important scalable, multi-parameter technique to the TDA toolbox.  The structure of the paper follows our main contributions.  
\begin{itemize}
    \item We define the Euler surface corresponding to bi-filtrations and higher-dimensional parameter spaces, and relate to current directions in TDA (Section~\ref{sec:basics}); 
    \item Give efficient algorithms for a variety of input including cubical and simplicial complexes arising from embedded point clouds along with the a Python package for computing the Euler surfaces (Section~\ref{sec:algorithms});
    \item Show that Euler curves and surfaces are useful for a variety of classification tasks both for data generated by random models (Section~\ref{sec:experiments}), as well as real-world medical data (retinal images) (Section~\ref{sec:realworld});
    \item Most importantly show how Euler surfaces can give insight into the structure of datasets by highlighting ``interesting" areas of the parameter space (Section~\ref{sec:realworld}). 
\end{itemize}

\subsection{Related work}
\label{sec:related_work}
This work can be seen as an attempt to understand multi-parameter filtrations, which avoids the difficulties inherent in multi-parameter persistence by considering a simpler topological invariant: namely the Euler characteristic. We do not recount the numerous approaches to multi-parameter persistence here as it is not directly applicable to this work and will be obvious to experts in the field, while we certainly lose quite a bit of information in using this simpler invariant, we gain a readily computable and applicable approach to data analysis.  

The connection of Euler characteristics and data analysis goes back to the Kac-Rice formula which gives the expected number of critical points of a sufficiently nice random field~\cite{adler2009random}. This is most naturally thought of as the study of a random function on a space. Taking the sublevel/superlevel sets of the random function yields a one-dimensional filtration, which in turn for every function gives a piecewise constant integer valued function -- this is called the Euler characteristic curve. As the input is random, it is natural to take the expectation.  Due to what can be understated as fortuitous, there exists a closed-form formula for the expected Euler characteristic curve, which is called the Gaussian kinematic formula. This  applies to wide range of spaces, see  ~\cite{adler2009random} for an in-depth account.

 This has been applied to fMRIs~\cite{worsley2004unified}, cosmology~\cite{van2011alpha}, and more recently various machine learning classification problems~\cite{effEC}. There has also been research into efficient streaming algorithms for their computation on image data  \cite{streamECC}. 
In terms of multi-parameter settings, we mention~\cite{adler2012rotation}, which proves certain convergence properties of a Euler surface arising from smoothing  a Gaussian random field (GRF).

Multiple Euler characteristic curves have also been used for shape analysis in the form of the Euler characteristic transform (ECT)~\cite{10.1093/imaiai/iau011,curry2018many,ghrist2018persistent}. In Section~\ref{sec:ect}, we comment and describe the relationship between the two constructions.   This can be seen indirectly as an alternative approach to multi-parameter settings, with both applications and interesting theoretical properties. The ECT is closely related to the Euler integration ~\cite{baryshnikov2011inversion,ghrist2011euler,baryshnikov2009target}, which like the Euler charateristic has appeared in several different mathematical areas.


\section{Preliminaries}
\label{sec:basics}
We begin with the definition of the Euler characteristic: 
\begin{definition}   
\label{def:euler_char}
Let $X$ be a CW complex of dimension $k$. The Euler characteristic $\chi(X)$ of $X$ is
\begin{equation}
    \label{eq:euler_char}
    \chi(X) = \sum_{i=0}^k (-1)^i  n_i,
\end{equation}
where $n_i = |\{ \sigma \in X \ \vert \ \sigma \text{ is } i  \text{ dimensional }\}|$, or the number of $i$-dimensional cells.
\end{definition}
There are many other ways to characterize the Euler characteristic, including the alternating sum of Betti numbers or the integral of curvature over a Riemannian manifold (or appropriate triangulation). As our primary interest is (finite) data, we do not further recount the basic properties of the Euler characteristic, referring the reader to any standard text on algebraic  topology~\cite{hatcher2002algebraic} or differential topology~\cite{guillemin2010differential}. Furthermore, we may assume that we are generally in a setting where different notions of Euler characteristic coincide. While our definition is in terms of CW complexes, we focus on two special cases which represent constructions which are either direct representations of data or are  common constructions from data. That is, we restrict ourselves to
\begin{enumerate}
    \item simplicial complexes 
    \item cubical complexes/images
\end{enumerate}
In the case of simplicial complexes, we consider primarily proximity based complexes such as  the alpha complex and the Vietoris-Rips complex. While we give different constructions, we generally denote the underlying complex by $K$.
 
\begin{remark}
While we use the corresponding structure of these special cases in the algorithms, there is a clear modification for more general complexes, e.g. cellular complexes, which may arise from other types of preprocessing such as collapses using Discrete Morse Theory. However, as the algorithms are linear in the size of simplicial complexes, it is likely that any additional preprocessing is likely to increase computation time, so we do not consider it here. 
\end{remark}


\begin{figure}[tb]
    \centering
    \begin{subfigure}[b]{1.7in}
        \includegraphics[width=\textwidth]{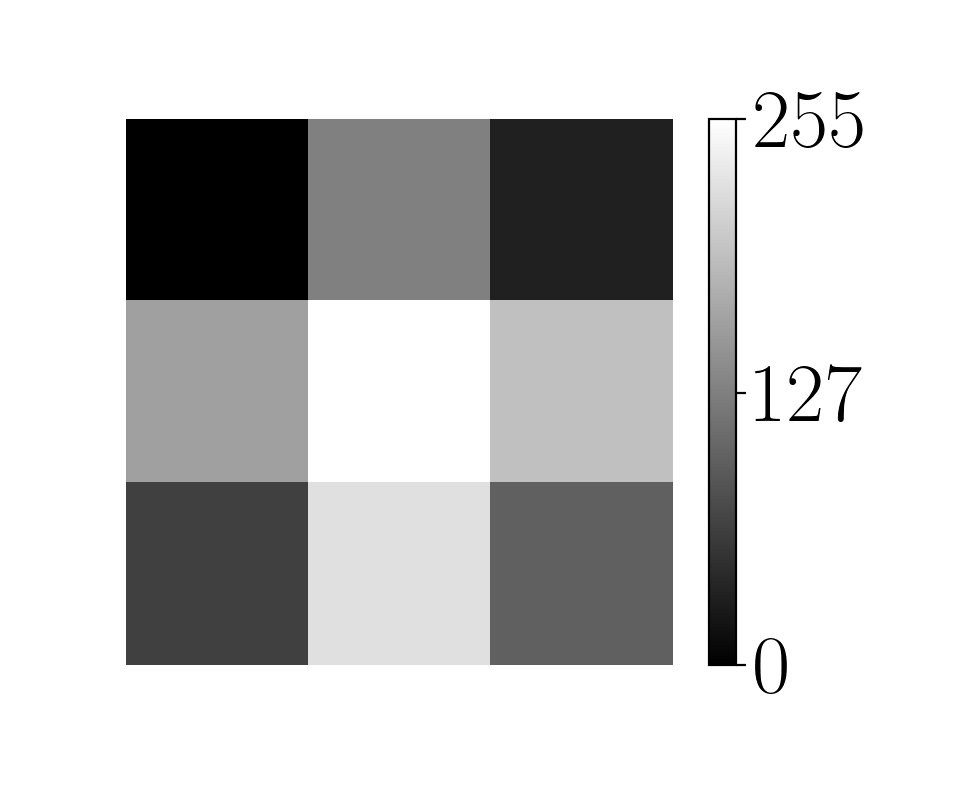}
        \caption{}   
        \label{fig:toy_image}
    \end{subfigure}
    \begin{subfigure}[b]{1.4in}
        \includegraphics[width=\textwidth]{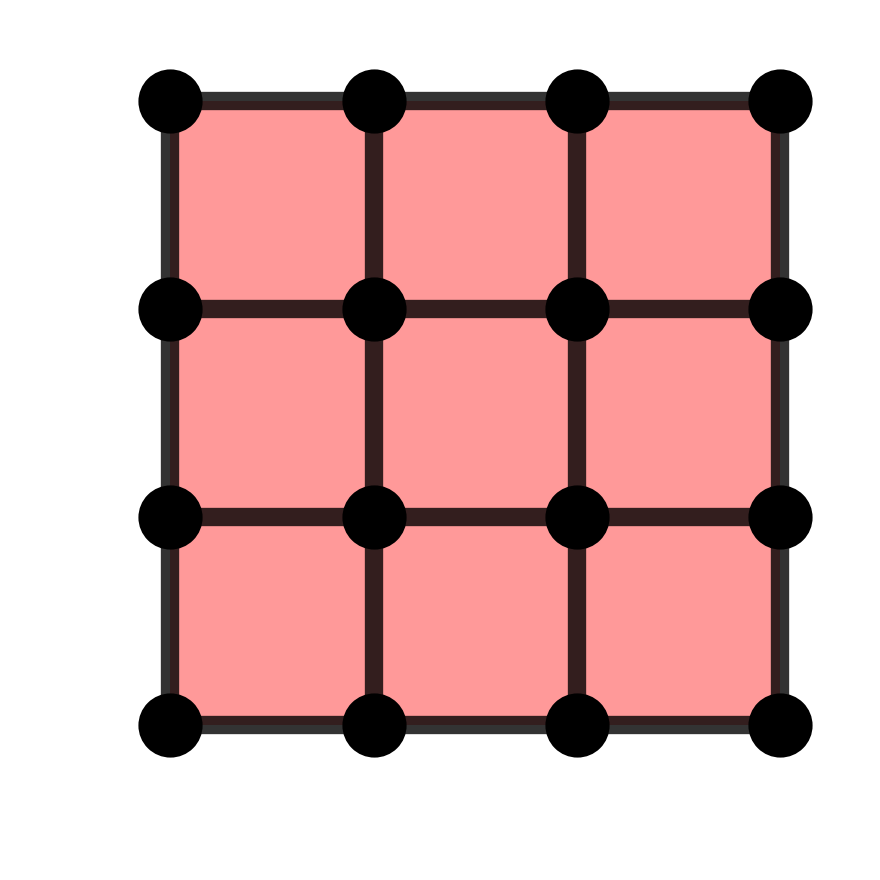}
        \caption{}   
        \label{fig:toy_cubical}
    \end{subfigure}
    \qquad
    \begin{subfigure}[b]{1.4in}
        \includegraphics[width=\textwidth]{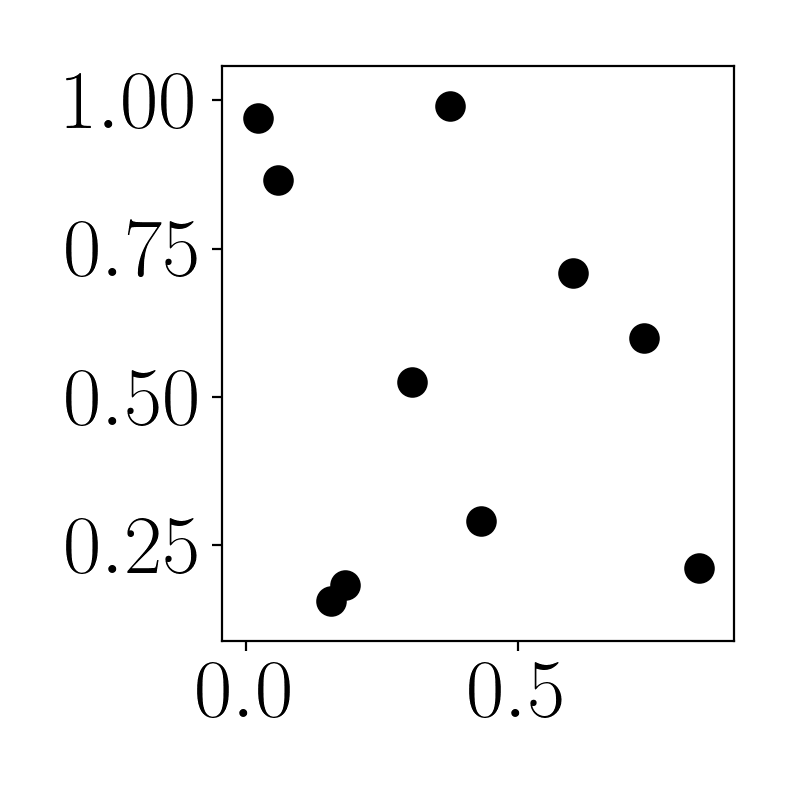}
        \caption{}   
        \label{fig:toy_points}
    \end{subfigure}
    \begin{subfigure}[b]{1.4in}
        \includegraphics[width=\textwidth]{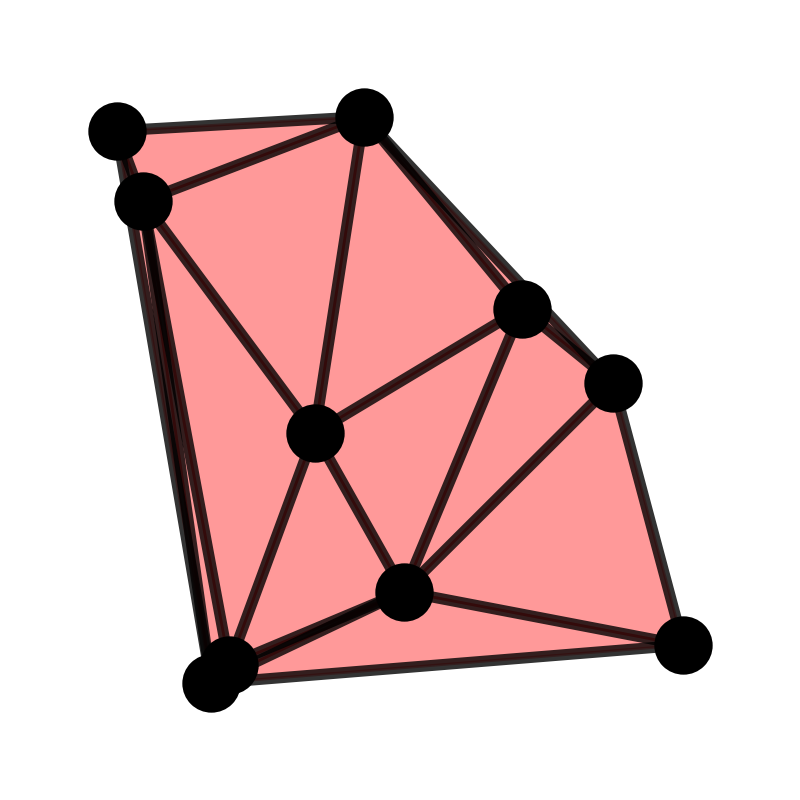}
        \caption{}   
        \label{fig:toy_delaunay}
    \end{subfigure}
    \caption{From left to right: $3\times 3$ gray-scale image, full cubical complex $Q$ of the image in (a), finite point set in $\mathbb{R}^2$, and Delaunay complex $D$ of the points in (c).}
    \label{fig:toy_data}
\end{figure}

\begin{figure}[tb]   
  \centering
  \begin{subfigure}[b]{4.5in}
    \fbox{\includegraphics[width=\textwidth]{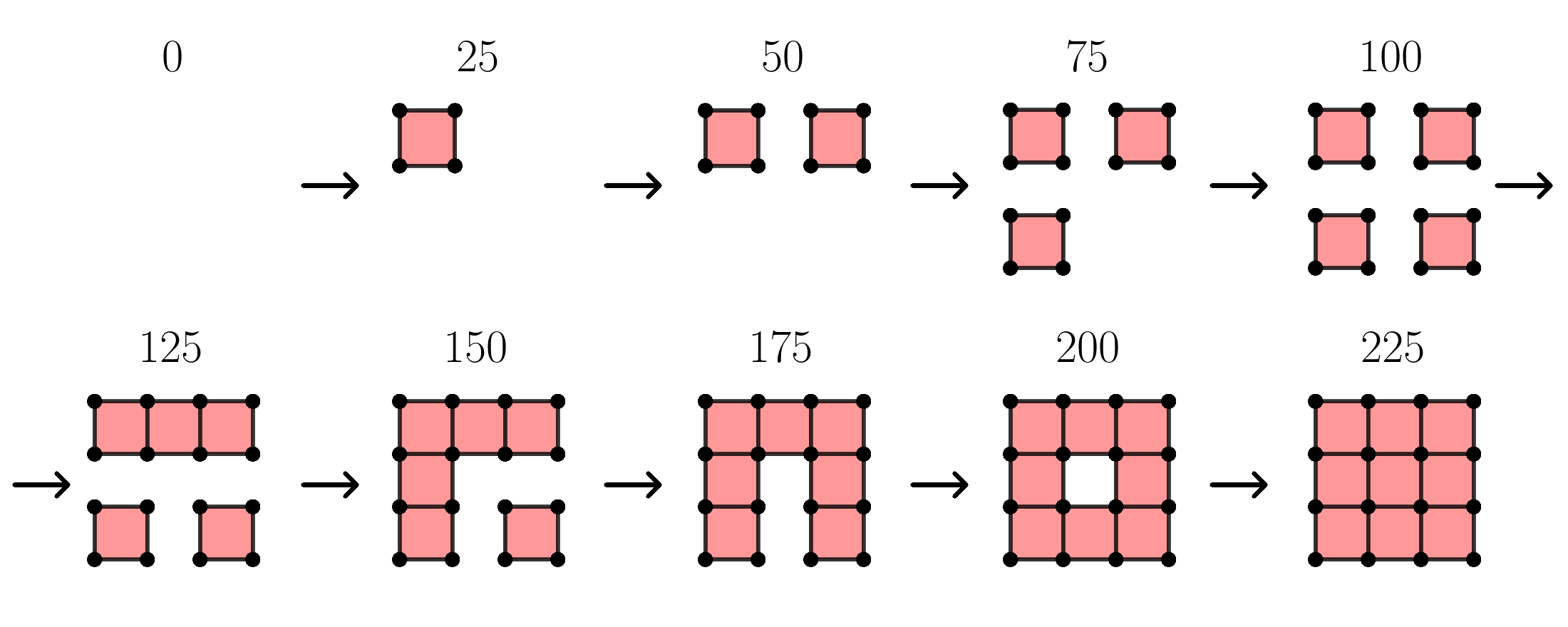}}
     \caption{Filtration of the gray-scale image in Figure \ref{fig:toy_image}.}
    \label{fig:toy_filtration_image}
  \end{subfigure}
  \begin{subfigure}[b]{4.5in}
    \fbox{\includegraphics[width=\textwidth]{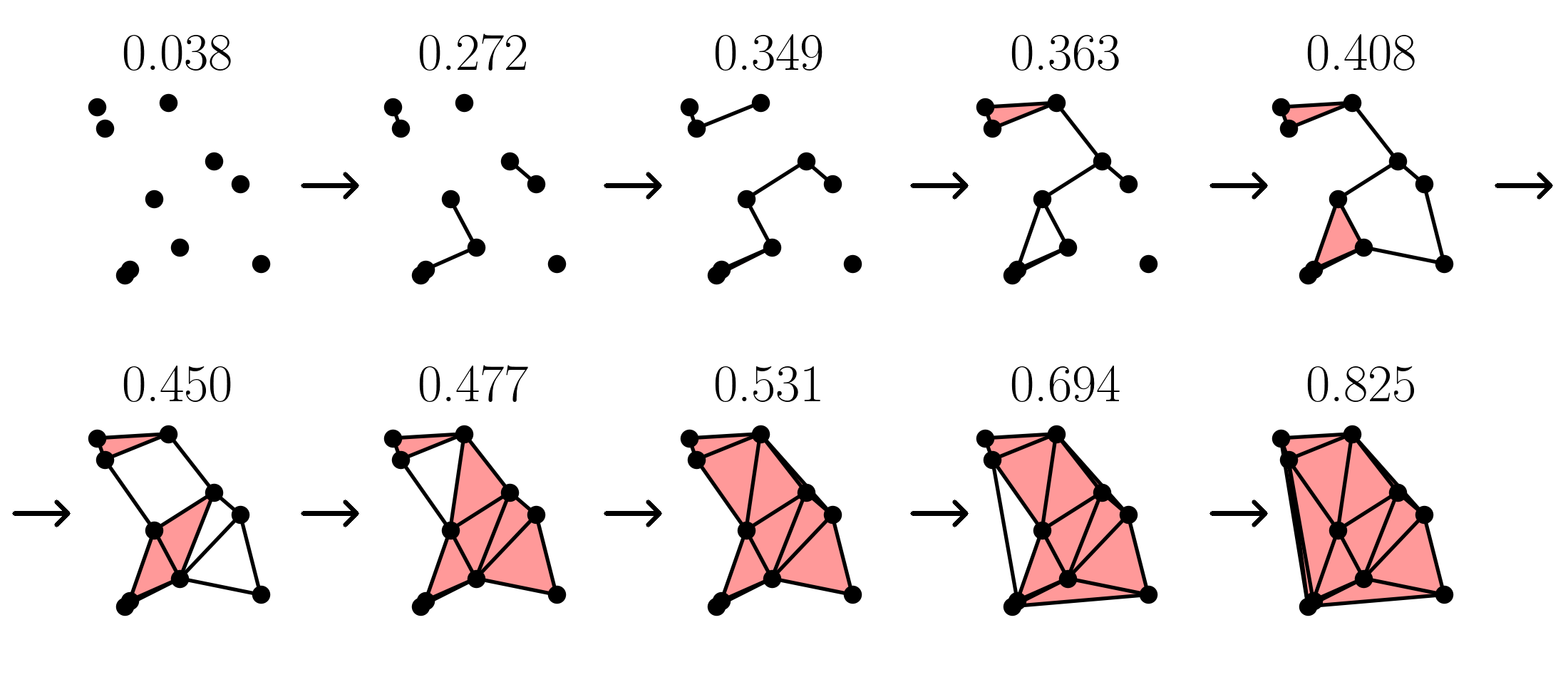}}
    \caption{Filtration of the finite set of points in Figure \ref{fig:toy_points}.}
    \label{fig:toy_filtration_points}
  \end{subfigure}
  \caption{Filtrations of the example data in Figure \ref{fig:toy_data}. In both (a) and (b) the filtration parameters are displayed above each subcomplex. The last subcomplexes in the sequence are the full cubical complex $Q$ and the Delaunay complex $D$.}
  \label{fig:toy_filtrations}
\end{figure}






We begin with the one-dimensional parameter case. Due to its connection with persistent homology, it is the most well-studied and familiar case in topological data analysis (TDA). The basic object of study is no longer one space, but a sequence of spaces called a \emph{filtration}.
A filtration is a increasing sequence of nested spaces:
\begin{equation}
    \emptyset \subseteq X_1 \subseteq X_2 \subseteq \cdots \subseteq X_m.
\end{equation}
The indexing set may be discrete ($\mathbb{Z}$) or continuous ($\mathbb{R}$). One of the most common ways to construct filtrations in TDA is as sub-level sets of functions. Given a real-valued function $f:X\rightarrow \mathbb{R}$ and a threshold $\alpha$, we may obtain a space
$$X_\alpha = f^{-1}((-\infty, \alpha]).$$ 
By varying $\alpha$, we obtain the \emph{sub-level set filtration} (resp. the super-level set filtration induced by $f^{-1}([\alpha,\infty)$) induced by $f$.
 In many cases of interest, such as piecewise-linear functions on simplicial complexes, these filtrations are topologically equivalent to a filtration induced by a function which is piecewise constant on each simplex. Namely for a cell, $\sigma \in K$, we have a function
 $$g(\sigma) = \max\limits_{x\in \sigma} f(x)$$
 Under this definition, the sublevel sets of the function forma filtration by subcomplexes since  $g(\tau) \leq g(\sigma)$ for every $\tau, \sigma \in K$ with $\tau$ a face of $\sigma$. 

Since our main interest lies in applications, we restrict ourselves to this piecewise constant setting, with an increasing sequence of finite complexes. In this case, we may assume without loss of generality, that the indexing set is discrete, although for exposition, we refer to the function value of a simplex as to when it enters the filtration.

When these restriction is valid has been well-studied, but we note that this is  the case well-behaved functions of finite complexes. For the interested reader, we point out that the classes of functions which are well-behaved include Morse and constructible functions~\cite{van1998tame,kashiwara1997integral}, but in many cases even these conditions may be relaxed. 

We make one further simplifying assumption: that cells enter the filtration one at a time. This is neither necessary nor restrictive but is rather done for the exposition of the algorithms. A simple condition which ensures this, is the assumption that all cells are assigned a unique function value. 

We refer to a piecewise constant function which induces a filtration as a \emph{filtering function}.
%
%
We now define the Euler characteristic curve.

\begin{figure}[tb]
    \centering
    \begin{subfigure}[b]{0.45\textwidth}
        \includegraphics[width=\textwidth]{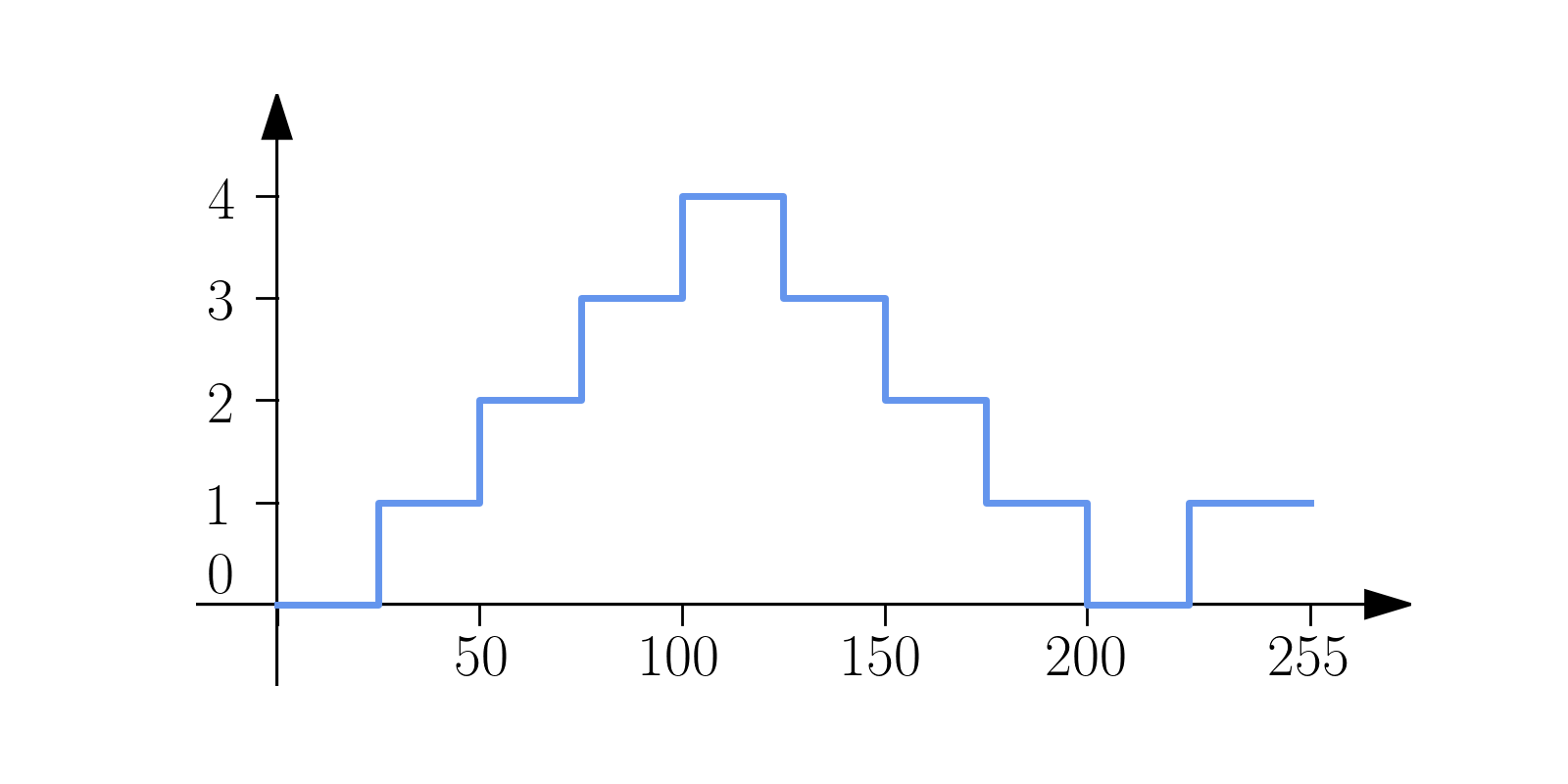}
        \caption{} 
        \label{fig:toy_euler_image}
    \end{subfigure}
    \begin{subfigure}[b]{0.45\textwidth}
        \includegraphics[width=\textwidth]{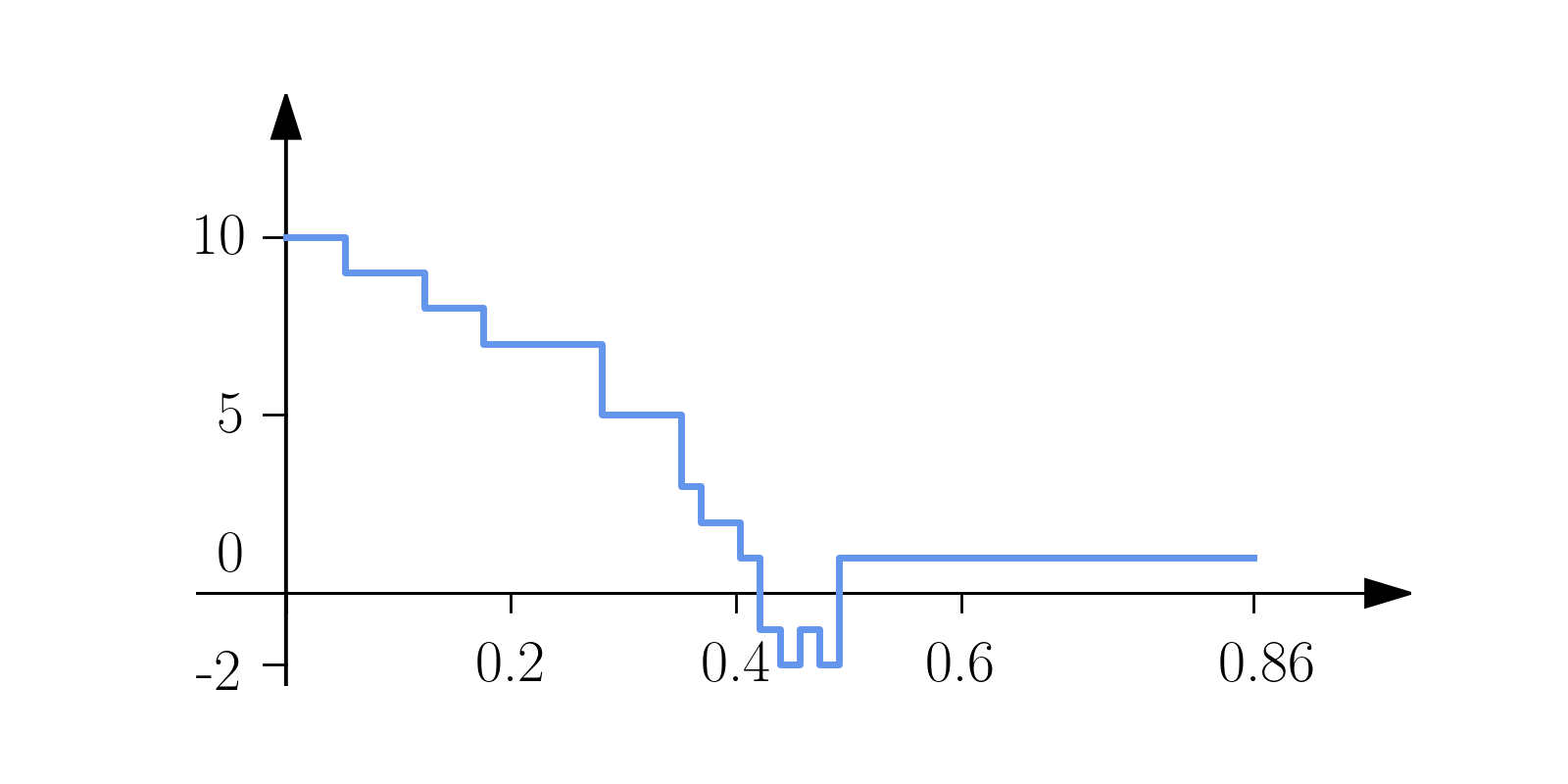}
        \caption{} 
        \label{fig:toy_euler_points}
    \end{subfigure}
    \caption{The plot in (a) is the Euler characteristic curve of the image in Figure \ref{fig:toy_image}, where sublevel sets are on values between $0$ and $255$. The plot in (b) is the Euler characteristic curve of the points in Figure \ref{fig:toy_points}, where the sublevel sets are on values $a_s$ such that there is a one simplex difference between $Q_{a_{s-1}}$ and $Q_{a_s}$ for each $0 \leq s \leq m_1$.}
    \label{fig:toy_euler_curves}
\end{figure}


\begin{definition}   
\label{def:euler_curve}
Given a filtering function $h: K \rightarrow \mathbb{R}$, the Euler characteristic curve induced by $h$ is the integer-valued function $E_h: \mathbb{R}\rightarrow \mathbb{Z}$ defined by
\begin{equation}
    E_h(\alpha) = \chi(f^{-1}((-\infty, \alpha])), 
\end{equation}
for each $\alpha \in \R$.
\end{definition}
Examples of the Euler characteristic curves of are given in Figure \ref{fig:toy_data} are in Figure \ref{fig:toy_euler_curves}. As described in Section~\ref{sec:related_work}, this curve has been used in a  number of difffernt applications. 

\section{Euler Characteristic Surfaces}
Our goal is to extend this to multi-parameter filtrations.  In generalizing, we consider filtrations which are Cartesian products of one parameter filtrations.
\begin{definition}
A $k$-parameter filtration is the Cartesian product of $k$ one-parameter filtrations.
\end{definition}
Our restriction to products may seem restrictive. However, any finite poset may be embedded into a product of linear orders, which is referred to as the order dimension or the Dushnik–Miller dimension. While in general, computing the minimal embedding dimension of of a poset is NP-hard, there are many special cases which are known. Constructive techniques for embedding posets into Cartesian products is interesting but we leave it for further work, as our main interest in applications come from Cartesian products induced by functions in $\R^d$ (although there are many cases of interest which do not fall into this category). 
In the interest of readability, we focus primarily the case of bi-filtrations (also since these are also readily visualised).
\begin{definition}   
\label{def:bifiltration}
Let $h_1: K \rightarrow \mathbb{R}$, $h_2: K \rightarrow \mathbb{R}$ be two filtering functions on a complex, and $\mathbf{h} = (h_1, h_2):K\rightarrow \mathbb{R}^2$ the function defined by $\mathbf{h}(\sigma) = (h_1(\sigma), h_2(\sigma))$ for each $\sigma \in K$. Given two sets of monotonically increasing values $\mathcal{R}_1 = \{ a_s \}_{s=0}^{m_1}$, $\mathcal{R}_2 = \{ b_t \}_{t=0}^{m_2}$ and defined $K_{s,t} = \mathbf{h}^{-1}((-\infty, a_s]\times (-\infty, b_t])$, the sublevel set bi-filtration of $K$ induced by $\mathbf{h}$ on $\mathcal{R}_1$, $\mathcal{R}_2$ is the grid of nested subcomplexes
\begin{equation}
    \begin{matrix}
    K_{0,0} & \subseteq & K_{0,1} & \subseteq & \cdots & \subseteq & K_{0,m_2} \\
    \rotatebox[origin=c]{270}{$\subseteq$} & &
    \rotatebox[origin=c]{270}{$\subseteq$} & & & &
    \rotatebox[origin=c]{270}{$\subseteq$} \\   
    K_{1,0} & \subseteq & K_{1,1} & \subseteq & \cdots & \subseteq & K_{1,m_2} \\
    \rotatebox[origin=c]{270}{$\subseteq$} & &
    \rotatebox[origin=c]{270}{$\subseteq$} & & & &
    \rotatebox[origin=c]{270}{$\subseteq$} \\   
    \vdots & & \vdots & & \ddots & & \vdots \\
    \rotatebox[origin=c]{270}{$\subseteq$} & &
    \rotatebox[origin=c]{270}{$\subseteq$} & & & &
    \rotatebox[origin=c]{270}{$\subseteq$} \\   
    K_{m_1,0} & \subseteq & K_{m_1,1} & \subseteq & \cdots & \subseteq & K_{m_1,m_2} 
\end{matrix}
\end{equation}
We say that $\mathbf{h}: K \rightarrow \mathbb{R}^2$ is a bi-filtering function on $K$.
\end{definition}
%
\begin{definition}  
\label{def:euler_surf}
Let $\mathbf{h}: K \rightarrow \mathbb{R}^2$ be a bi-filtering function on $K$ and $\mathcal{R}_1 = \{a_s \}_{s=0}^{m_1}$, $\mathcal{R}_2 = \{ b_t \}_{t=0}^{m_2}$ two set of monotonically increasing real values. The Euler characteristic surface induced by $\mathbf{h}$ on $\mathcal{R}_1$, $\mathcal{R}_2$ is the matrix of integer values
\begin{equation}
    S_{\mathbf{h}} = 
    \begin{pmatrix}
    \chi(K_{0,0}),& \chi(K_{0,1}),& \cdots & \chi(K_{0,m_2}) \\
    \chi(K_{1,0}),& \chi(K_{1,1}),& \cdots & \chi(K_{1,m_2}) \\
    \vdots & \vdots & \ddots & \vdots \\
    \chi(K_{m_1,0}),& \chi(K_{m_1,1}),& \cdots & \chi(K_{m_1,m_2})
    \end{pmatrix},
\end{equation}
where $K_{s,t} = \mathbf{h}^{-1}((-\infty, a_s]\times (-\infty, b_t])$.
\end{definition}

The last column and last row of $S_{\mathbf{h}}$ are the Euler characteristic curves of $h_1$ and $h_2$ respectively. So the Euler characteristic surface contains all the topological information of $E_{h_1}$ and $E_{h_2}$ at the values in $\mathcal{R}_1$ and $\mathcal{R}_2$, plus the information coming from intersection of sublevel sets of $h_1$ and $h_2$.  
    
In practice, we have two forms of data we consider: digital images, which are treated as a form of cubical complex, and simplicial complexes, where we focus on proximity complexes on point clouds. 
\begin{definition}   
\label{def:filter_image}
Let $M$ be a $n_1$-by-$n_2$ two-dimensional gray-scale image.
The cubical complex $Q$ of $M$ is defined by the union of squares $C_{i,j} = [i, i+1] \times [j, j+1]$ and their subfaces for each pixel $(i,j)$ of $M$.
We define the pixel intensity filtering function $h_M: Q \rightarrow \mathbb{R}$ setting 
\begin{enumerate}
    \item[\textit{(i)}] $h_M(C_{i,j}) = v_{i,j}$  for each $1 \leq i \leq n_1$ and $1 \leq j \leq n_2$, where $v_{i,j}$ is the intensity of pixel $(i,j)$;
    \item[\textit{(ii)}] $h_M(C) = \min\limits_{C \subseteq C_{i,j}} h_M(C_{i,j}) $ for each element $C$ of $Q$.
\end{enumerate}
In the same way, given a three-dimensional gray-scale image $M$, $h_M$ is defined by voxels intensities $v_{i,j,k}$ setting 
$h_M(C_{i,j,k}) = v_{i,j,k}$ 
and 
$h_M(C) = \min\limits_{C \subseteq C_{i,j,k}} h_M(C_{i,j,k})$ 
for each voxel $(i,j,k)$ and each element $C$ in $Q$.
\end{definition}

We also consider simplicial complexes built on point sets in $\mathbb{R}^d$.  
We focus on \emph{proximity complexes} -- complexes which depend on distances between the points. 
These include common constructions in TDA such as \v Cech complexes, Vietoris-Rips complexes, alpha and Delaunay complexes~\cite{CompTop2010}.
%
%

%
Again we consider the image $M$ in Figure \ref{fig:toy_image}, and the finite set of points $X$ in Figure \ref{fig:toy_points}. The filtrations induced on this data by the filtering functions $h_M$ and $h_X$ are in Figure \ref{fig:toy_filtration_image} and Figure \ref{fig:toy_filtration_points} respectively. These functions may be based on distance or some other function(s). In our experiments, we use the Delauanay complex as most of our experiments are in two or three-dimensional space. 
\begin{remark}
    A key advantage to using the Euler characteristic is that it is defined pointwise. That is, the value is determined by each complex, the maps between the spaces have no effect on the invariant, i.e. different maps will lead to the same Euler characteristic surface. While this does make it a weaker invariant, it also implies that the problems present in multidimensional persistence are avoided. 
\end{remark}

\paragraph{Example: Euler characteristic surfaces of random images.} 
We show with an example that the Euler characteristic surfaces of pairs of images can contain useful information for distinguishing between different classes in a dataset, while the Euler characteristic curves of the same images do not.

Our data consists of families of pairs of $n_1 \times n_2$ random images $M_1$, $M_2$, whose pixel intensities are uniformly distributed and have an expected correlation $p$. 
These are generated by drawing three sample values for each pixel $(i,j)$. 
In particular, we draw $x, v_1, v_2$ from independent uniform distributions $\mathcal{U}(0,1)$, $\mathcal{U}(0, 256)$, $\mathcal{U}(0,256)$. 
If $x \leq p$,  we set $M_1[i][j] = M_2[i][j] = \lfloor v_1 \rfloor$. 
Otherwise, we set $M_1[i][j] = \lfloor v_1 \rfloor$ and $M_2[i][j] = \lfloor v_2 \rfloor$.
Hence, we set $M_1[i][j]$ and $M_2[i][j]$ to the same random integer with probability $p$ and to independently drawn random integers with probability $(1-p)$.

Given a pair of gray-scale images $M_1$, $M_2$, we obtain an Euler characteristic surface by bi-filtering on the pixel intensity filtering functions $h_{M_1}$, $h_{M_2}$.
In this setting, we can derive the expected value of the Euler characteristic $\chi(Q_{s,t}) = \chi\left( \mathbf{h}^{-1}((-\infty, s] \times (-\infty, t]) \right)$ for each pair of thresholds $0 \leq s \leq m_1$, $0 \leq t \leq m_2$.
We know that a vertex $(i,j)$ is in $Q_{s,t}$ if and only if at least one of the squares that include it is in $Q_{s,t}$.
The same holds for edges in $Q_{s,t}$.
Thus, given the probability of having squares in $Q_{s,t}$, the probabilities of having vertices and edges can be derived.
By the definition of $M_1$ and $M_2$, in terms of random values sampled from uniform distributions, it follows that
\begin{equation}
    \begin{split}
      P(C_{i,j} \in Q_{s,t}) = & \
		    \ P \Big( h_{M_1}(C_{i,j}) < s \text{ and } 
		             h_{M_2}(C_{i,j}) < t
		             \text{, with } h_{M_1}(C_{i,j}) = h_{M_2}(C_{i,j}) \Big) \cdot p  \\
		    & + P \Big( h_{M_1}(C_{i,j}) < s \text{ and } 
		                h_{M_2}(C_{i,j}) < t 
		                \text{, with } h_{M_1}(C_{i,j}), h_{M_2}(C_{i,j}) \text{ independent} \Big) \cdot (1-p) \\
	        = & \ \ \min \{ s, t\} \cdot p 
	              + s \cdot t \cdot (1-p)	 
    \end{split}
\end{equation}
where $0 \leq s, t \leq 255$ and $C_{i,j}$ is any square in the two-dimensional cubical complex $Q$.
Then, because the values of different pixels are independent of each other, the probability that a vertex/edge $\sigma'$ belongs to $Q_{s,t}$ is $1 - ( 1 - P(C_{i,j} \in Q_{s,t})^k)$, where $k$ is the number of squares containing $\sigma'$.

Finally, there are $n_1 \cdot n_2$ squares in the cubical complex $Q$ of $M_1$ and $M_2$.
These contain $(n_1+1)\cdot(n_2+1)$ vertices, subdivided into $(n_1-1) \cdot (n_2-1)$ internal vertices contained into $4$ squares each, $2(n_1-1)+2(n_2-1)$ boundary vertices contained into $2$ squares each, and $4$ corner vertices contained in a single square.
Moreover, there are $n_1(n_2+1) + n_2(n_1+1) -2n_1 -2n_2$ interval edges contained in $2$ squares each, and $2n_1 + 2n_2$ boundary edges contained in a single square.
Hence the expected value of $\chi(Q_{s,t})$ is
\begin{equation}
    \label{eq:exp_eu_char}
	\begin{split}
		\mathbb{E}[\chi(Q_{s,t})] & = (n_1-1) \cdot (n_2-1) \cdot 
		                 [1 - (1 - P(C_{i,j} \in Q_{s,t})^4)] 
		                 \\
					 	 & + (n_1 \cdot (n_2+1) + n_2\cdot (n_1+1) - 4) \cdot 
					 	 [1 - (1 - P(C_{i,j} \in Q_{s,t})^2)] 
					 	 \\
						 & + (n_1 \cdot n_2 + 2n_1 + 2n_2 + 4) \cdot 
						 P(C_{i,j} \in Q_{s,t}).
	\end{split}
\end{equation}

\begin{figure}[tb]
    \centering
    \begin{subfigure}[b]{0.32\textwidth}
        \includegraphics[width=\textwidth]{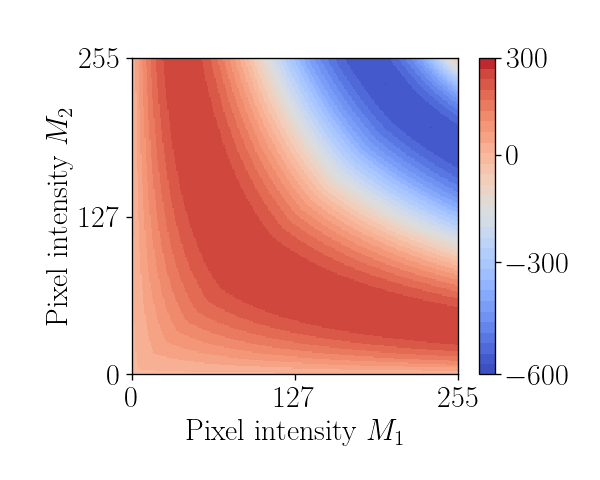}
        \caption{} 
        \label{fig:analytical-surf-1}
    \end{subfigure}
    \begin{subfigure}[b]{0.32\textwidth}
        \includegraphics[width=\textwidth]{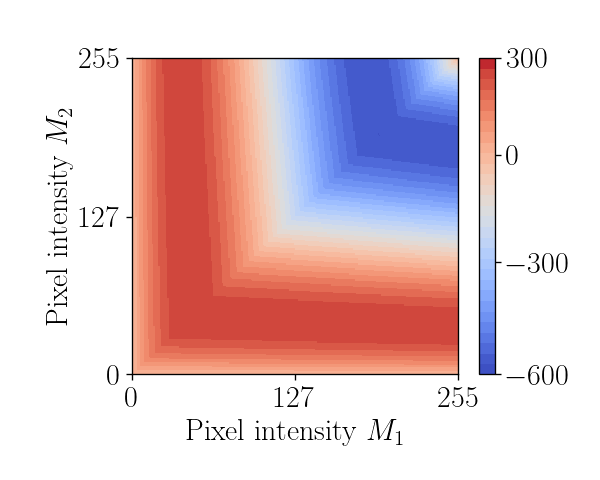}
        \caption{} 
        \label{fig:analytical-surf-2}
    \end{subfigure}
    \begin{subfigure}[b]{0.32\textwidth}
        \includegraphics[width=\textwidth]{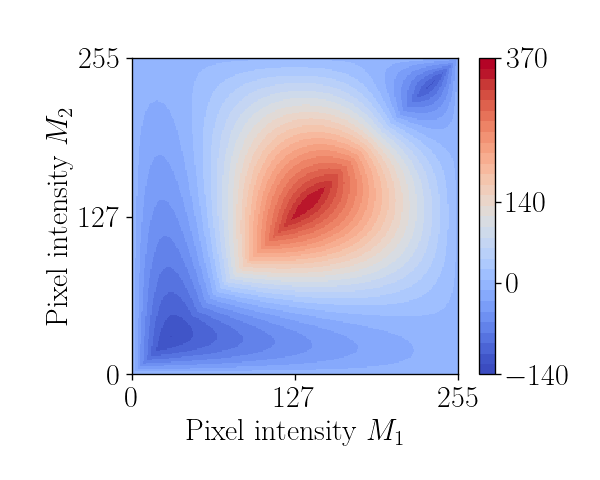}
        \caption{} 
        \label{fig:diff-analytical-surf}
    \end{subfigure}
    \caption{Contour plot of $S_{(h_{M_1}, h_{M_2})}$ of a pair of random images $M_1$, $M_2$ with correlation coefficient equal to $0.1$ in (a) and equal to $0.8$ in (b). The difference between these two Euler surfaces is the Euler terrain in (c).}
    \label{fig:expected_surfaces}
\end{figure}

Figure \ref{fig:expected_surfaces} displays the contour plots of expected Euler characteristic surfaces for two different values of $p$. The images defining them have the same uniformly random distribution of pixel values, so the corresponding expected Euler characteristic curves are all equal. 
However, the Euler characteristic surfaces are able to discriminate between pairs of images with different correlation $p$, see the Euler terrain in Figure \ref{fig:diff-analytical-surf}.
Furthermore, it is readily apparent that the structure of the of surfaces is encoding information about the correlation. Since this is a toy model, we do not characterize the behaviour rigorously here. However, it does serve as an indication as to the structural information which is encoded in Euler surfaces. This can be particularly useful for comparing models with real-world data, as the shape of the surface may indicate what behavior in data, a model is failing to capture. This is especially intriguing as closed form expressions often exist Euler characteristics \cite{adler2009random,bobrowski2020homological} or alternatively, as a well-behaved functional, it can be empirically estimated. 
\begin{problem}
 How Euler surfaces for models with multiple parameters differ for varying values of  parameters? 
\end{problem}

\myparagraph{Difference of Euler surfaces.} Expected Euler surfaces are most often well-behaved  with closed-form solutions often existing. Alternatively the empirical expected Euler characteristic can be efficiently computed as the surfaces are defined pointwise. Therefore, when comparing Euler surfaces which come from two classes, a natural object to investigate is the \emph{difference of average Euler surfaces}. While this is computed pointwise, the resulting structure of this object provides insight into parameter values which differentiate between the two classes. 
\begin{definition}
    Given two sets of Euler surfaces over a common parameter space denoted $\{K_{s,t}\}$ and $\{K'_{s,t}\}$, their difference at point $s,t$ is defined as
    $$ \mathrm{ECS} (K,K')[s][t] := \frac{1}{|\{K_{s,t}\}|}  \sum \chi(K_{s,t}) - \frac{1}{|\{K'_{s,t}\}|}  \sum \chi(K'_{s,t}) $$
\end{definition}
For brevity, we refer to this object as the {\bf Euler terrain} or simply as the {\bf terrain}. One important point for this object is that if we are comparing two distinct random processes where the expected Euler characteristic exists and the corresponding surface is continuous (after scaling), replacing the empirical averages with expectations results in the expected difference. By linearity of expectations, the terrain also exists, is continuous, and provides a comparison of two processes. 

In experimental instances where the number of samples for the two classes is highly unbalanced (or the number of samples is small), it makes sense to further normalize by the standard deviation of each Euler characteristic. We refer to this as a \emph{normalized terrain}, given by  
  $$ \overline{\mathrm{ECS}} (K,K')[s][t] := \frac{ECS(K,K')[s][t]}{\mathrm{sd}(\chi(K)[s][t]) +\mathrm{sd}(\chi(K')[s][t])}$$
where $\mathrm{sd}$ is the standard deviation. Terrains provide interesting information about the data which we investigate experimentally in Sections ~\ref{sec:experiments} and ~\ref{sec:realworld}.

\section{Connection with Euler Characteristic Transform}
\label{sec:ect}
The Euler Characteristic Transform (ECT) \cite{10.1093/imaiai/iau011,curry2018many,ghrist2018persistent} was originally considered for three-dimensional shape analysis.  Here  we present the basic form and refer the reader to \cite{curry2018many} for more general and in-depth presentation. 

Given a compact shape in $\R^d$, consider the space of directions given as points on $\mathbb{S}^{d-1}$. For each direction consider the sub-level set filtration given by the height function in that direction.  This yields an Euler Characteristic Curve for each direction.
There is the surprising theorem:
\begin{theorem}[Theorem 3.4~\cite{curry2018many}]
 If $M$ and $M'$ are two constructible subsets of $\R^d$, then equivalence of the  Euler characteristic transforms implies that the sets are equal,
 i.e.
 $$\mathrm{ECT}(M) =\mathrm{ECT}(M') \Rightarrow M=M'$$ 
\end{theorem}
This theorem essentially implies that we do not lose any information when passing to a collection of  Euler characteristic curves.
There has been substantial interest toward understanding how finite sampling affects reconstruction. 

A natural question is how do Euler surfaces relate to the ECT. First, the ECT uses many filtering functions. In its theoretical set-up, for a shape in $\R^d$, a function is constructed for each point in $\mathbb{S}^{d-1}$. Even in the approximate case, many different directions are needed. As we are primarily interested in visualizing the results, we restrict to only a few functions (most often only two). Hence the first (rather obvious) connection is that from an ECT, the choice of any $k$ directions will result in a $k$-dimensional Euler Surface (or more accurately an Euler $k$-dimensional hypervolume).

In our applications, there is often a natural real-valued function, i.e. pixel intensity, along with other parameters. The above theorem also applies to this setting.
\begin{corollary}
Given a function on a subset  $K\subset \R^d$, the graph of the function can be reconstructed by the ECT.
\end{corollary}
\begin{proof}
Since the domain of the function is a subset of $\R^d$, the graph can be embedded as a shape in $\R^{d+1}$. The above theorem then immediately implies the result.
\end{proof}
One observation we make is that in the above case, we need not take all directions on $\mathbb{S}^d$. Rather, we can consider each levelset of the function separately. Each levelset is a set in $\R^d$, which can be reconstructed by considering directions in $\mathbb{S}^{d-1}$. Note that this does not reduce the number of directions considered but can be more intuitive when we are studying real-valued functions.

\begin{figure}[tbp]
    \centering
    \includegraphics[width=0.65\textwidth]{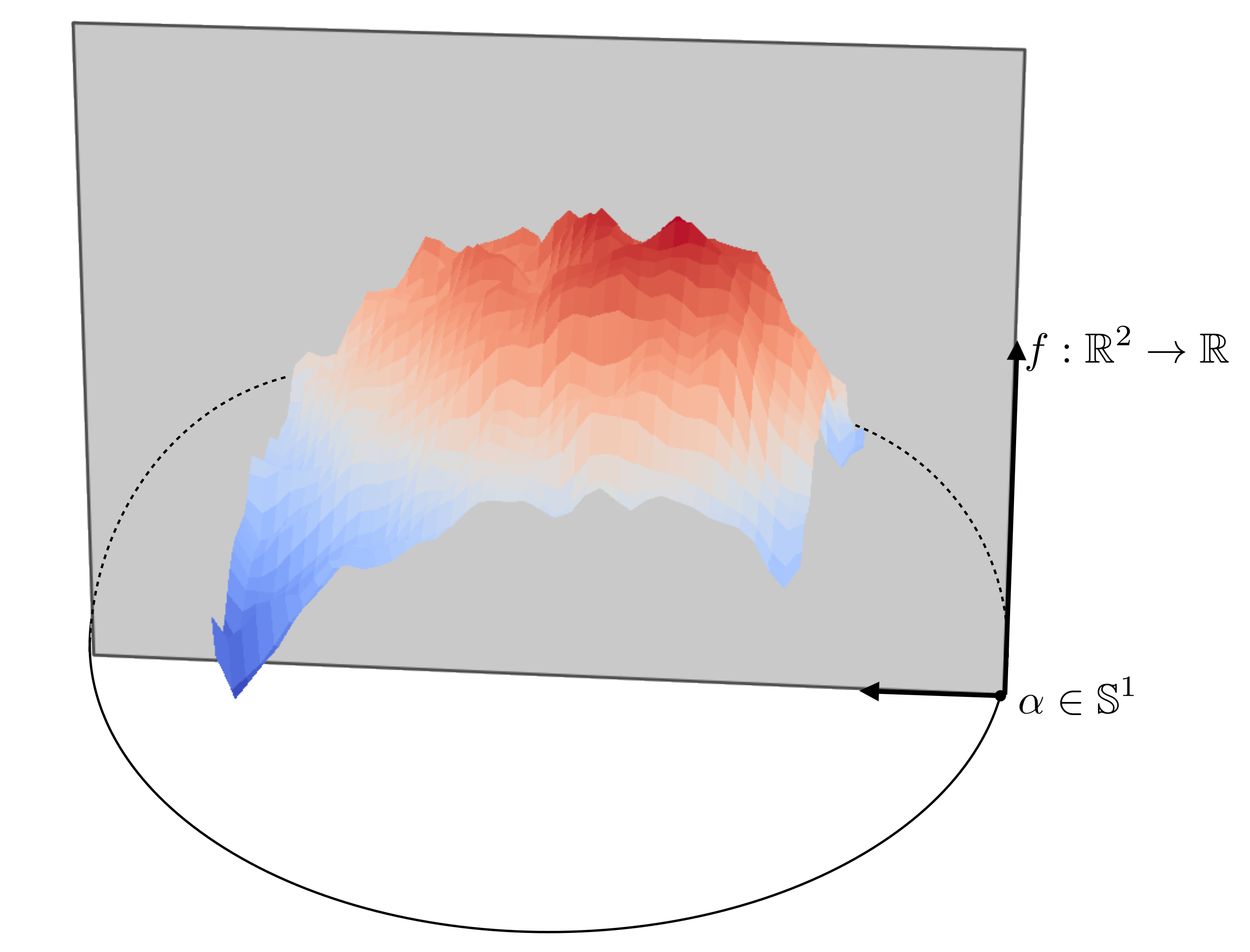}
    \caption{The realtionship between the ECT and the ECS for functions on $\mathbb{R}^2$. For every direction $\alpha$, the ECT computes the Euler characteristic curve based on the filtration arising from the height function in that direction. In the standard setup, it would be a direction on the sphere $\mathbb{S}^2$, but it could equivalently be all ``directions'' in $\mathbb{S}^1\times \mathbb{R}$. The 2-dimensional Euler surface we consider is a direction and the function sublevel sets, so the Euler surface can be thought of as representing the slice shown.  }
    \label{fig:ect1}
\end{figure}
A second observation is that this argument can be iterated, allowing us to reduce to the case of reconstructing shapes in $\R^2$ with directions in $\mathbb{S}^1$. While we believe this reduction can be useful in further analysis of the ECT, in this setting, the Euler Surfaces we use are a point in $\mathbb{S}^1$ for each function (see Figure~\ref{fig:ect1}).  This perhaps best illustrates the connection, where the Euler Surfaces are a sampling of the ECT. As we will show, despite not characterising the shape completely as the ECT, the Euler surfaces can still show us useful information, so we present the following open problem:
\begin{problem}
Given an ECT, is possible to automatically choose some set of ``interesting" Euler surfaces which can help us understand the underlying structure?
\end{problem}
In addition to being highly useful for data analysis, this can provide some better insight into the theory of ECTs.

\section{Algorithms}
\label{sec:algorithms}

In this section, we describe efficient algorithms for the computation of Euler characteristic surfaces of image and point data. 
An implementation of these is provided by the \href{https://pypi.org/project/euchar/}{euchar} Python package, the source code of which is available on Github at \href{https://github.com/gbeltramo/euchar}{https://github.com/gbeltramo/euchar}.

\subsection{Cubical Complexes/Image Data}
\label{subsec:image_data}
First, we describe an algorithm for the computation of the Euler characteristic surface of a pair of gray-scale images $M_1, M_2$. 
In particular, Algorithm \ref{alg:images} returns the surface of the bi-filtration of the sublevel sets of $\mathbf{h}: Q \rightarrow \R^2$ defined by $\mathbf{h}(C) = \left(h_{M_1}(C), h_{M_2}(C)\right)$ for each $C \in Q$, where $h_{M_1}$ and $h_{M_2}$ are the pixel intensity filtering functions of $M_1$ and $M_2$ respectively.

\paragraph{Discussion.} By definition of Cartesian product, we have that
$Q_{s,t} 
= 
\bfh^{-1} \big( (-\infty, a_s] \times (-\infty, b_t] \big)$ 
is equivalent to 
$Q_{s,t} = 
h_{M_1}^{-1} \big((-\infty, a_s] \big) \cap h_{M_2}^{-1} \big( (-\infty, b_t] \big)$.
So each column of the Euler characteristic surface $S_\bfh$ equals the Euler characteristic curve of $h_{M_1}$ with $Q$ restricted to its top-dimensional cubes $C$ such that $h_{M_2}(C) \leq b_t$, because of the intersection with the cubical complex $h_{M_2}^{-1}\big((-\infty, b_t] \big)$.
Thus, a n\"aive approach for computing the $S_\bfh$ is to compute Euler characteristic curves multiple times with algorithms looping on the cubes in $Q$ \cite{streamECC}.
To improve over this, Algorithm \ref{alg:images} makes use of the following two strategies:
\begin{itemize}
    \item[\emph{(i)}] Precompute the possible Euler characteristic changes produced by adding a top-dimensional $C$ into any $Q_{s,t}$, and use these to increase or decrease the values of $S_\bfh$;
    \item[\emph{(ii)}] Loop on each top-dimensional $C$ only once, by modifying all columns of $S_\bfh$ where $C$ produces the same change at the same time.
\end{itemize}
In the following discussion, points \emph{(i)} and \emph{(ii)} above are shown to preserve the correctness of the na\"ive approach computing columns of $S_\bfh$ independently.

Using Euler characteristic changes as suggested in \emph{(i)} is possible because the process of going from the empty abstract cubical complex to $Q = Q_{255,255}$ can be decomposed into steps at which a single $\bar{C}$ and its subfaces are added.
This follows from the definition of the filtering functions $h_{M_1}$ and $h_{M_2}$ in terms of pixel (voxel) intensity values.
Furthermore, at each such step, the change $\change$ in Euler characteristic of the current cubical complex is completely determined by the structure of elements adjacent to $\bar{C}$.
More precisely, defined the \emph{neighbourhood} $\neigh$ of $\bar{C}$ to be the set of cubes that intersect it, by Definition \ref{def:euler_char} $\change$ only depends on the numbers of cubes added into $\neigh$ when $\bar{C}$ is added.

\begin{figure}[tb]
\centering
\begin{minipage}{0.9\textwidth}
\begin{algorithm}[H]\caption{Euler characteristic surface of images.}
	\label{alg:images}
	 \hspace*{\algorithmicindent}{\textbf{Input:} gray-scale images $M_1, M_2$, $\mathbf{h}: Q \rightarrow [0, m_1] \times [0, m_2] \subseteq \mathbb{R}^2$, and the pre-computed vector \textit{preCompChanges}.}
	 \\
	    \begin{algorithmic}[1] 
	    \caption{Euler characteristic surface of bi-filtration on a pair of images.}
		\STATE Add a one pixel (voxel) thick outer layer to images, so that the new boundary pixels (voxels) are mapped by $\bfh$ into $(m_1+1, m_2+1)$
		\STATE $S_\bfh \gets\ (m_1+1)\times (m_2+1)$ zeros matrix
		\FOR{each top-dimensional cube $C$ in $Q_{M_1}$}
		    \STATE $a_s, b_t \gets h_{M_1}(C), h_{M_2}(C)$
  	    	\STATE $neigh_1, neigh_2 \gets h_{M_1}, h_{M_2}$ values in neighbourhood of $C$ 
  		    \STATE $thresholds_2 \gets$ sorted values in $neigh_2$ greater than $b_t$, union $m_2+1$
  		    \STATE $N_1^{C} \gets$ boolean matrix defined by $(neigh_1 \leq a_s)$ before ${C}$ and $(neigh_1 < a_s)$ after ${C}$
  		    \FOR{$k = 1$ to $|thresholds_2|$}
  			 	\STATE $N_2^{C} \gets$ boolean matrix defined by $(neigh_2 \leq thresholds_2[k-1])$
  			 	\STATE $N^{C} \gets$ element-wise AND of $N_1^{C}$ and $N_2^{C}$
  			 	\STATE $l \gets$ decimal integer of binary representation of $N^{C}$
  			 	\FOR{$\hat{t}=$ index of $thresholds_2[k-1]$ to index of $thresholds_2[k]-1$}
  			 	    \STATE $S_\bfh[s][\hat{t}]$ += \textit{preCompChanges}[$l$]
  			 	\ENDFOR
  			 	
  			 \ENDFOR
  		\ENDFOR 
  		\STATE $S_\bfh \gets$ cumulative sum on columns of $S_\bfh$
  		\STATE \textbf{return} $S_\bfh$
	    \end{algorithmic}
	\end{algorithm}
\end{minipage}
\end{figure}

All possible Euler characteristic changes can be precomputed because there is a finite number of neighbourhoods $\neigh$.\footnote{
For two-dimensional images, $\neigh$ is a set of $8$ squares and their subfaces, while for three-dimensional images it is a set of $26$ cubes and their subfaces.}
In particular, there are $2^{(3^d-1)}$ such neighbourhoods in dimension $d$, meaning that there are $256$ Euler characteristic changes to precompute for two-dimensional images and $67,108,864$ changes for three-dimensional images.
For $d=4$, the number of possible neighbourhoods is already a $25$ digits integer, making the computation and storage of their corresponding changes impractical.
Hence Equation \eqref{eq:euler_char} can be used to compute all the Euler characteristic changes for $d=2$ and $d=3$, which can then be stored in a vector $preCompChanges$ using the binary representation of neighbourhoods to index them.
For example, consider the neighbourhood in Figure \ref{fig:local1} corresponding to the binary matrix
\begin{equation}
    \begin{pmatrix}
        1 & 0 & 1 \\
        0 & 0 & 0 \\
        1 & 0 & 1
    \end{pmatrix},
\end{equation}
and in turn to the binary sequence $10100101$.
Its Euler characteristic change is $-3$ and the decimal representation of its binary sequence $165$.
Thus $-3$ is stored as the $165$-th element of $preCompChanges$.

Point $\emph{(ii)}$ above is realized by the inner loop on lines $8-15$ of Algorithm \ref{alg:images}, where $a_s=h_{M_1}(\bar{C})$ and $b_t=h_{M_2}(\bar{C})$ so that $Q_{s,t}$ is the first complex including $\bar{C}$.
The idea is to use $preCompChanges$ to update the $s$-th row of $S_\bfh$ at each iteration.
This can be done because $Q_{s,t} = h_{M_1}^{-1} \big((-\infty, a_s] \big) \cap h_{M_2}^{-1} \big( (-\infty, b_t] \big)$, so $\chi(Q_{s,t})$ and $\chi(Q_{s,t+1})$ can differ by a change $\change$ induced by $\bar{C}$ if and only if $\neigh$ in $Q_{s,t+1}$ has changed, i.e. if there is a top-dimensional cube $C' \in N^{\bar{C}}$ such that $h_{M_2}(C') = b_{t+1}$.
But all such changes depend on the $h_{M_2}$ values of top-dimensional cubes in $\neigh$ greater than $b_t$.
Sorting and storing these in $thresholds_2$ with $m_2+1$ appended, it follows that the ranges of $\hat{t}$-th columns of $S_\bfh$ such that $\hat{t}$ is between two consecutive values of $thresholds_2$ are such that the Euler characteristic change induced by adding $\bar{C}$ is constant because $\neigh$ does not change.
So the elements of vector $preCompChanges$ can be used on line $13$ to update all $\hat{t}$ columns such that $\hat{t} \geq j$.

In conclusion, at the end of the loop on lines $3-16$, each entry $S_\bfh[s][t]$ equals the change $\chi(Q_{s,t}) - \chi(Q_{s-1,t})$, because all changes $\change$ induced by the top-dimensional $\bar{C}$ in $Q_{s,t} \setminus Q_{s-1,t}$ have been considered.
After the cumulative sum on columns of $S_\bfh$, it follows that
\begin{equation}
    \begin{aligned}
    \label{eq:cumulative_sum}
        S_\bfh[s][t]
        = & \Big( \chi(Q_{0,t}) - 
                  \chi(\emptyset) 
            \Big) 
          + \ldots 
          + \Big( \chi(Q_{s,t}) - \chi(Q_{s-1,t}) \Big) 
        \\
        = & 
        \chi(Q_{s,t}) - \chi(\emptyset)
        = 
        \chi(Q_{s,t}),
    \end{aligned}
\end{equation}
which is the required Euler characteristic surface entry.

Ignoring the time required to precompute the vector of Euler characteristic changes $preCompChanges$, Algorithm \ref{alg:images} has a worst case running time of $O(n m_2 + m_1 m_2 )$, where $n$ is the number of pixels (voxels) in $M_1$ and $M_2$.
This follows because the inner loop on lines $8-15$ takes $O(m_2)$ operations in the worst case to update an entire row.
However, compared to computing $m_2+1$ Euler characteristic curves as proposed by the na\"ive approach at the beginning of this discussion, $\neigh$ is computed only once for ranges of columns where it does not change, and entries $S_\bfh$ are incremented and decremented without having to count subfaces of top-dimensional cubes in $\neigh$.
\begin{figure}[hbt]
    \centering
    \begin{subfigure}[b]{0.45\textwidth}
        \includegraphics[width=\textwidth]{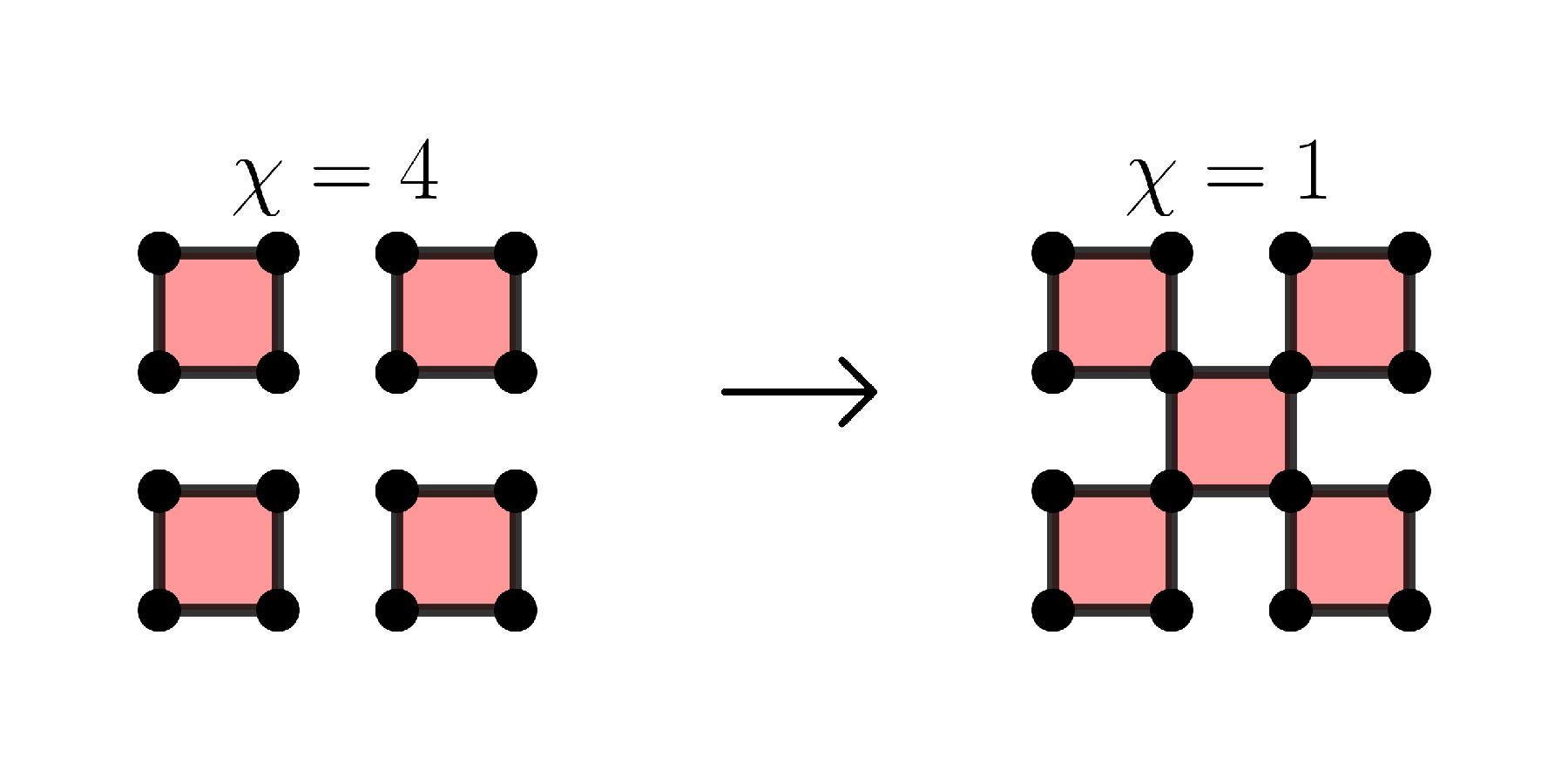}
        \caption{} 
        \label{fig:local1}
    \end{subfigure}
    \quad
    \begin{subfigure}[b]{0.45\textwidth}
        \includegraphics[width=\textwidth]{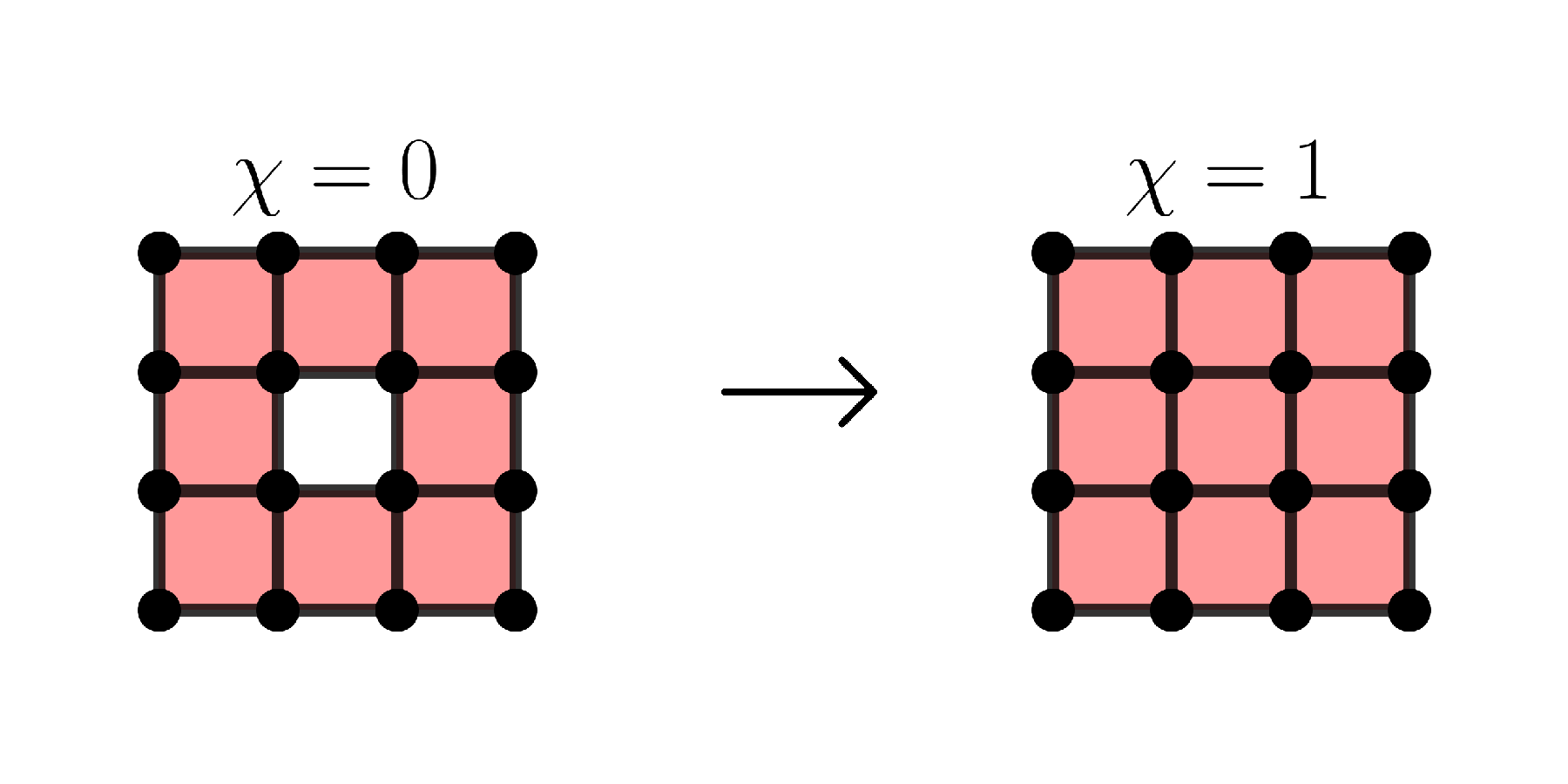}
        \caption{} 
        \label{fig:local2}
    \end{subfigure}
    \caption{Euler characteristic changes produced by adding a cube of maximal dimension in a two-dimensional cubical complex $Q_{s-1, t}$. In (a) the change is equal to $-3$, while in (b) it is $+1$.}
    \label{fig:local}
\end{figure}

\begin{figure}[bt]
\centering
\begin{minipage}{\textwidth}
\begin{algorithm}[H]
    \caption{Euler characteristic surface of bi-filtration on finite point set.}
	\label{alg:points}
	\hspace*{\algorithmicindent}{\textbf{Input:} abstract simplicial complex $K$, $\bfh = (h_1, h_2): K \rightarrow \mathbb{R}^2$, and sorted values in $\mathcal{R}_1$ and $\mathcal{R}_2$.} 
	 \\
	\begin{algorithmic}[1]
		\STATE $S_\bfh \gets$ $(m_1+1) \times (m_2+1)$ zeros matrix
		\FOR{each simplex $\sigma$ in $K$}
		    \STATE $v_1, v_2 \gets$  $h_1(\sigma)$, $h_2(\sigma)$
		    \STATE $a_s, b_t \gets$ minimum values greater than $v_1$, $v_2$ in $\mathcal{R}_1$, $\mathcal{R}_2$ with binary search
		    \FOR {$\hat{j} = j$ to $m_2$}
		    \STATE $S_\bfh[s][\hat{t}] \gets (-1)^{dim(\sigma)}$
		    \ENDFOR 
		\ENDFOR
		\STATE $S_\bfh \gets$ cumulative sum on columns of $S_\bfh$
		\STATE \textbf{return} $S_\bfh$
	\end{algorithmic}
\end{algorithm}
\end{minipage}
\end{figure}

\subsection{Point Data}
\label{subsec:point_data}

Given a finite point set $X$, which we assume being in general position, we provide Algorithm \ref{alg:points} for the computation of the Euler characteristic surfaces of a bi-filtration of a simplicial complex $K$ built onto $X$.
Moreover, it is assumed that this bi-filtration consists of sublevel sets of a $\bfh = (h_1, h_2): K \rightarrow \R^2$ on monotonically increasing sets of real values $\mathcal{R}_1$ and $\mathcal{R}_2$.

\paragraph{Discussion.} In this case, when a simplex $\sigma$ is added into a $K_{s,t} = h_1^{-1} \big((-\infty, a_s] \big) \cap h_2^{-1} \big( (-\infty, b_t] \big)$ its neighbourhood does not have a fixed structure.
Thus it is not possible to precompute Euler characteristic changes as in Algorithm \ref{alg:images}.
However, if $\sigma \in K_{s,t}$, then $\sigma \in K_{s, \hat{t}}$ for each $\hat{t} \geq t$.
So the change in Euler characteristic $(-1)^{dim(\sigma)}$, produced by adding $\sigma$ into $K_{s,t}$, also applies to $K_{s, \hat{t}}$ for each $\hat{t} \geq t$.
This property is used on line $6$ of Algorithm \ref{alg:points} to update the $s$-th row of $S_\bfh$ for each $\sigma$.
It follows that at the end of the loop on lines $2-8$ each entry $S_\bfh[s][t]$ equals $\chi(K_{s,t}) - \chi(K_{s-1, t})$, and the cumulative sum on columns of on line $9$ returns the desired Euler characteristic surface.

Differently from Algorithm \ref{alg:images} (where value $a_s$ is mapped to index $s$, and $b_t$ to $t$), it is necessary to find the indexes $s,t$ of $h_1(\sigma)$ and $h_2(\sigma)$ within the sorted values of $\mathcal{R}_1$ and $\mathcal{R}_2$.
With a binary search this operation takes $\log_2(m_1)$ and $\log_2(m_2)$ operations for $s$ and $t$ respectively.
Hence the worst-case running time of Algorithm \ref{alg:points} is  $O(n (log_2(m_1) + log_2(m_2) + m_2) + m_1 m_2)$, where $n$ is the number of simplices in $K$.

In the point set case, there is no inherent locality that can be exploited. In some cases, one could construct the complex locally, but ultimately this is an improved technique of constructing the complex rather than any improvement in computing the Euler characteristic, which counts the simplices as a sorted list.
We conclude this section with a remark.
\begin{remark}
The algorithms presented here are quite straightforward and we include them primarily for completeness. This simplicity also leads to them being exceptionally efficient. One important open question is whether algorithms can be made sublinear if we allow for approximations. This is particularly important for higher-dimensional parameter spaces, as the complexity rises exponentially in the dimension of the parameter space. 
\end{remark}

\section{Simulated \& Random Data}
\label{sec:experiments}
In this section, we first present the use of Euler surfaces for the classification for images and differenting between various random processes including random images and point processes. In Section ~\ref{sec:realworld}, we look at a real-world dataset, but here we are able to investigate the resulting surfaces when the underlying process is known. 

\subsection{Gray-scale Image Classification}
\label{subsec:img_clf}

Here we consider two standard benchmarking image classification databases: the \texttt{OUTEX\_TC\_00000} test suite \cite{ojala2002outex} and the MNIST database of handwritten digits \cite{lecun1998mnist}.
 
 \begin{figure}[htb]
    \centering
    \begin{subfigure}[b]{0.45\textwidth}
        \includegraphics[width=\textwidth]{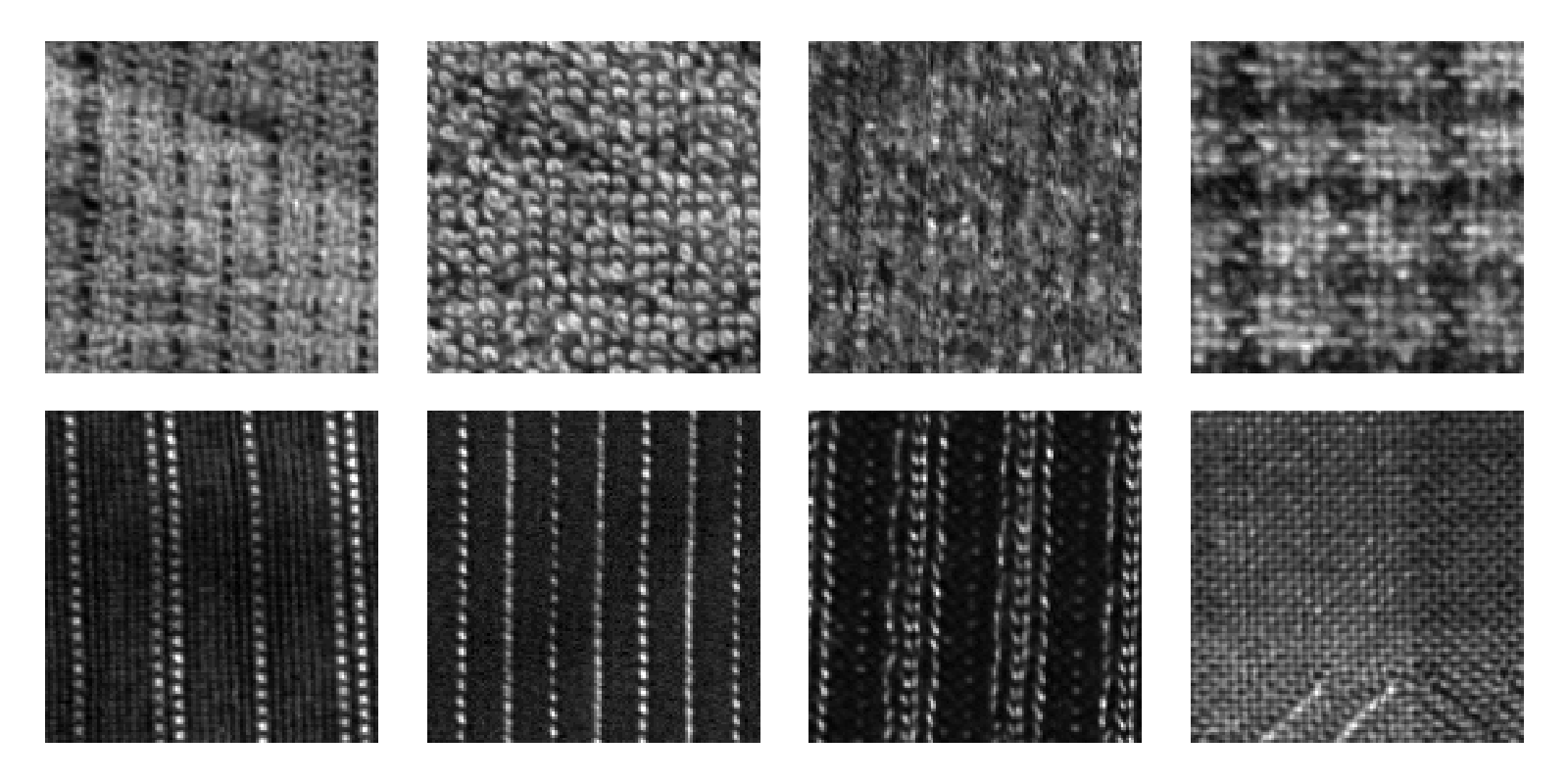}
        \caption{} 
        \label{fig:outex}
    \end{subfigure}
    \qquad
    \begin{subfigure}[b]{0.45\textwidth}
        \includegraphics[width=\textwidth]{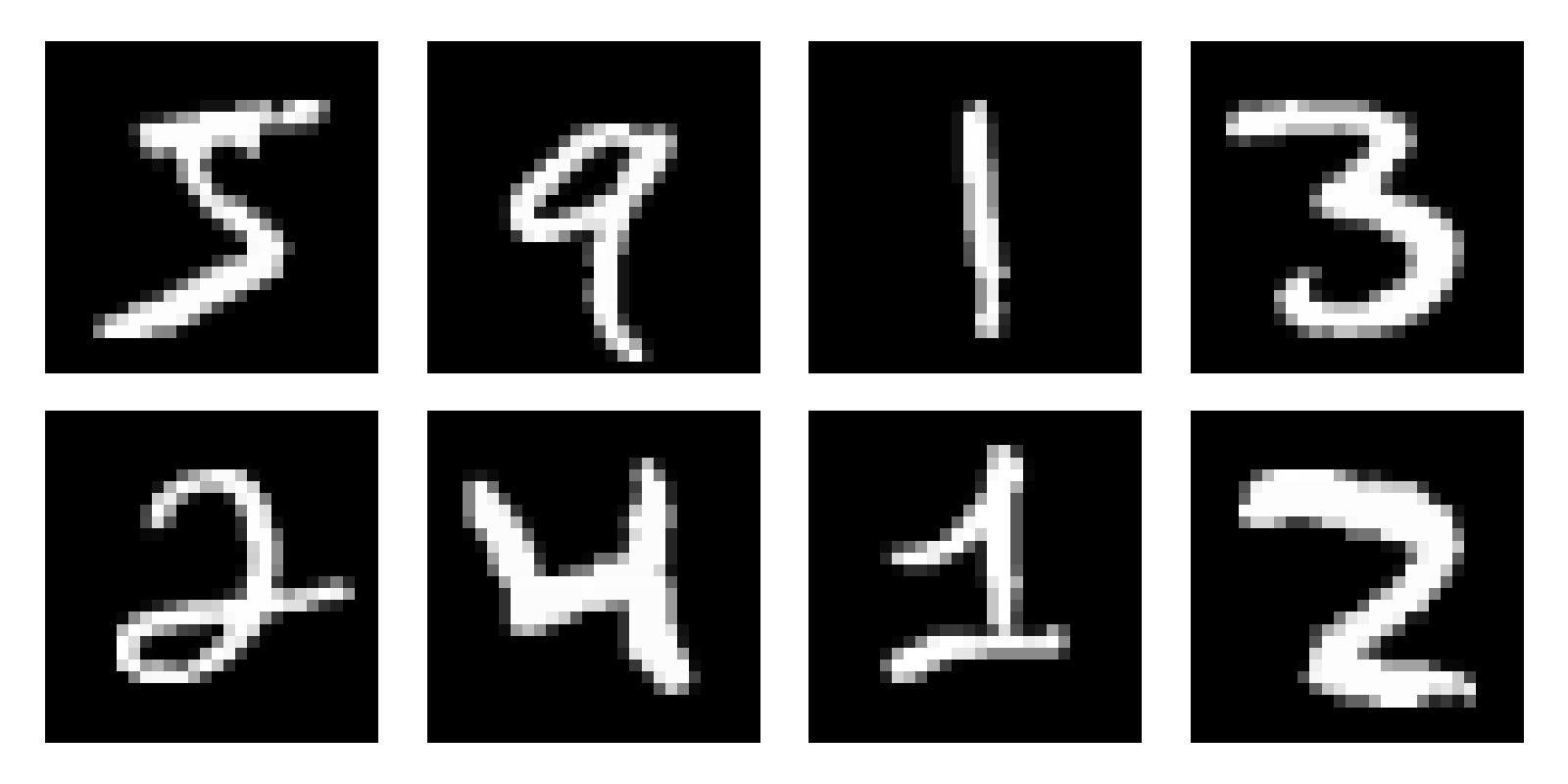}
        \caption{} 
        \label{fig:mnist}
    \end{subfigure}
    \caption{The images in (a) are eight of the twenty-four patterns in the \texttt{OUTEX\_TC\_00000} test suite. The images in (b) are a random selection of the $70,000$ handwritten digits from the MNIST database.}
\end{figure}

\begin{table}[tb]
    \caption{Classification results obtained with Euler characteristic based feature vectors and logistic regression.}
    \centering
    \begin{tabular}{lr}
        \toprule
        \multicolumn{2}{c}{\texttt{Outex\_TC\_00000}} \\
        \cmidrule(r){1-2}
        \textbf{Features} & \textbf{Avg test accuracy} \\
        \midrule
        Euler curve - pixel intensity & $91.08\pm 1.57$ \%  \\
        Euler surface - pixel intensity, Laplacian   & $\mathbf{96.29}\pm 1.24$ \% \\
        \bottomrule
        \toprule
        \multicolumn{2}{c}{MNIST} \\
        \cmidrule(r){1-2}
        \textbf{Features} & \textbf{Test accuracy} \\
        \midrule
        Euler curve - pixel intensity & $33.63$ \%  \\
        Euler surface - pixel intensity, top/bottom gradient   & $\mathbf{71.79}$ \% \\
        \bottomrule
    \end{tabular}
    \label{table:image_classify}
\end{table}

We test the performance of Euler characteristic curves versus Euler characteristic surfaces as feature vectors of logistic regression.
For this we use the implementation provided by the \texttt{scikit-learn} Python machine learning package \cite{pedregosa2011scikit}, together with a \texttt{lbfgs} solver. 
Each Euler characteristic curve is calculated directly on each gray-scale image $M_1$. On the other hand, to compute Euler characteristic surfaces we define a second image $M_2$ for each one in \texttt{Outex\_TC\_00000} and MNIST respectively.
For the former dataset, we set $M_2$ equal to the Laplacian of $M_1$, defined as the discrete convolution of $M_1$ with the kernel
$$\begin{pmatrix}
0 & -1 & 0 \\
-1 & 4 & -1 \\ 
0 & -1 & 0
\end{pmatrix}.$$
Instead, for the MNIST dataset of images, we define $M_2$ as the constant gray-scale image given by the top-down gradient of the same size of any MNIST image. Since MNIST images are $28\times 28$ in size, we have that $M_2[i][j] = \lfloor 255 \cdot \frac{i}{28} \rfloor$.

To obtain feature vectors for logistic regression, we subsample both Euler characteristic curves and surfaces by preserving only one element in $6$ for curves, and one row and column in $6$ for surfaces. Also, we concatenate the resulting rows of Euler characteristic surfaces. Finally the components of the resulting feature vectors are normalized to have mean $0$ and standard deviation $1$. 
The results are given in Table \ref{table:image_classify}, where we use average test accuracy as the scoring metric.
In both cases feature vectors derived from Euler characteristic surfaces result in higher classification scores.
By combining multiple sources of information Euler surfaces outperform Euler curves.
So, in setting in which data can be parameterized by more that one real-valued function, Euler surfaces are useful for encoding more information than Euler curves into feature vectors employed in classification tasks.

\subsection{Random Three-Dimensional Images}
\label{subsec:3d_images}

\begin{figure}[tb]
    \centering
    \begin{subfigure}[b]{0.4\textwidth}
        \includegraphics[width=\textwidth]{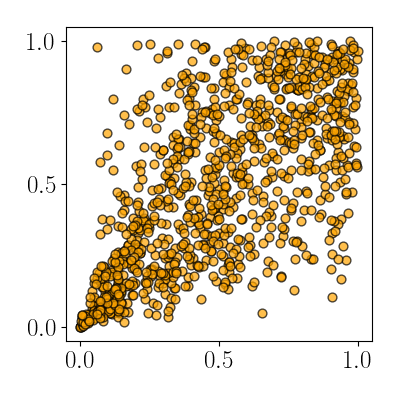}
        \caption{} 
        \label{fig:copula1}
    \end{subfigure}
    \qquad
    \begin{subfigure}[b]{0.4\textwidth}
        \includegraphics[width=\textwidth]{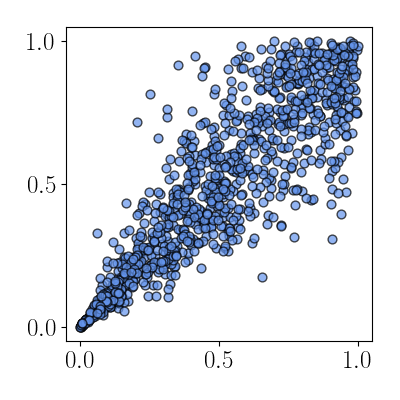}
        \caption{} 
        \label{fig:copula5}
    \end{subfigure}
    \caption{Points sampled from two Clayton copula distributions. In (a) $\theta = 1$, while in (b) $\theta = 5$.}
    \label{fig:copula_points}
\end{figure}
\begin{figure}[tb]
    \centering
    \centering
    \begin{subfigure}[b]{0.32\textwidth}
        \includegraphics[width=\textwidth]{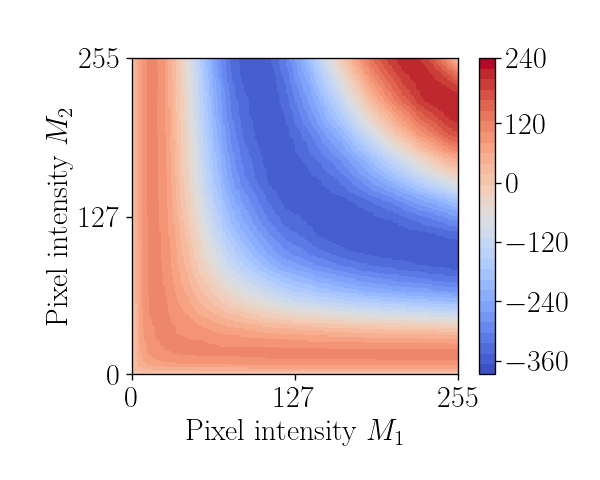}
        \caption{}
        \label{fig:copula-avg-surf-1}
    \end{subfigure}
    \begin{subfigure}[b]{0.32\textwidth}
        \includegraphics[width=\textwidth]{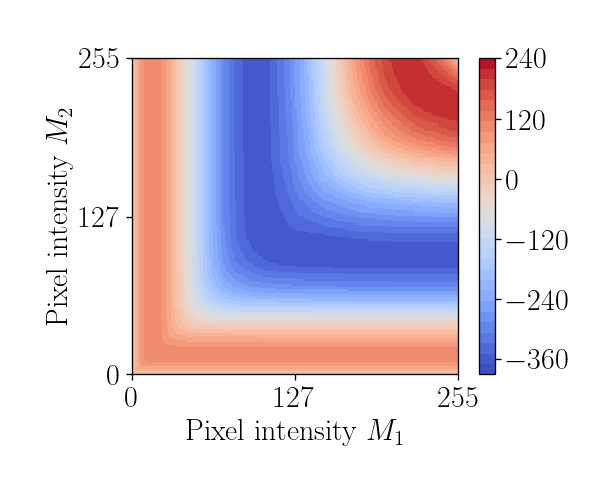}
        \caption{}
        \label{fig:copula-avg-surf-5}
    \end{subfigure}
    \begin{subfigure}[b]{0.32\textwidth}
        \includegraphics[width=\textwidth]{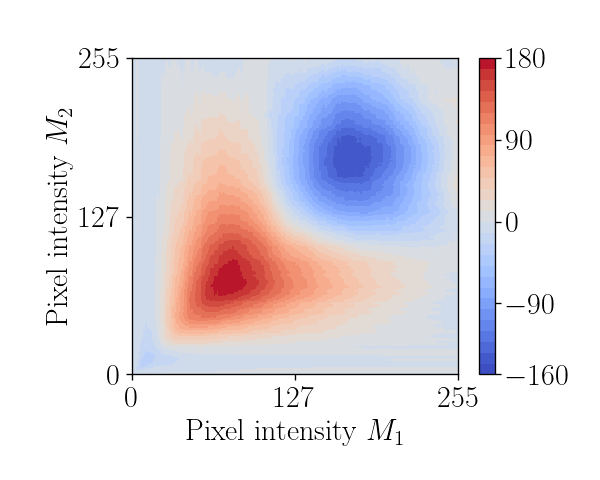}
        \caption{}
        \label{fig:copula-avg-surf-diff}
    \end{subfigure}
    \caption{Average Euler characteristic surfaces of random images generated with random points as in Figure \ref{fig:copula_points}. In (a) surface of images obtained with $\theta = 1$, and in (b) with $\theta=5$. The absolute value of the difference of the surfaces in (a) and (b) is in (c).}
    \label{fig:avg_copula_surf}
\end{figure}
The aim of this experiment is to provide empirical evidence of the conclusions of the example given in Section~\ref{sec:basics}. There we saw how the expected Euler characteristic surfaces of pairs of uniformly random images are able to distinguish between different levels of correlation $p$.
Here, we synthetically generate correlated pairs of uniformly random three-dimensional images. To have a well defined notion of correlation we employ copula distributions \cite{nelsen2007introduction}. 
In particular, we sample random points from a two-dimensional Clayton copula with uniform marginals $\mathcal{U}(0,256)$ as implemented by the \texttt{copula} R package, see \cite{hofert2018elements}. An example is given in Figure \ref{fig:copula_points}, where $1,000$ points were sampled from Clayton copula distributions with different $\theta$ parameters.

We make a pair of 3D random images $M_1$, $M_2$ of shape $n\times n \times n$, by mapping the coordinate values of $n^3$ points sampled from a Clayton copula in $\mathbb{R}^2$ to the voxel intensities of $M_1$, $M_2$.
Each $\theta$ results into a family of pairs of random 3D images.

We generate $100$ random pairs of 3D images with this method for both a Clayton copula distribution with $\theta = 1$ and with $\theta = 5$. 
The resulting $100$ Euler characteristic surfaces are averaged to obtain the surfaces in Figure \ref{fig:copula-avg-surf-1} and Figure \ref{fig:copula-avg-surf-5}. 
The absolute difference of the two is in Figure \ref{fig:copula-avg-surf-diff}.
As in the example in Section \ref{sec:basics} the Euler characteristic curves of any $M_1$ and $M_2$ are known to be have the same expected values.
On the other hand, the Euler terrain in Figure \ref{fig:copula-avg-surf-diff} clearly shows that the average Euler surfaces of different families of images distinguish between different values of $\theta$.

\subsection{Point Processes}
\label{subsec:point_proc}
Here we study the properties of Euler characteristic surfaces of finite point sets. We start by providing an example using Euler characteristic surfaces to encode the statistic properties of different point processes in $\mathbb{R}^2$. In particular, we consider 
\begin{itemize}
    \item[\textit{(i)}] A Poisson point process in the unit square $[0, 1]\times [0,1]$ of intensity $\lambda = 400$;
    \item[\textit{(ii)}] A Hawkes cluster process, as described in Section $3$ of \cite{kroese2013spatial}, whose cluster centers are generated as a Poisson process of intensity $\lambda = 280$. We take a offspring intensity function 
    $$\rho(x, y) = \frac{\alpha}{2\pi\sigma^2} \cdot e^{-\frac{1}{2\sigma^2} (x^2 + y^2)},$$
    with $\alpha = 0.3$ and $\sigma = 0.02$. Note that by the definition of Hawkes process and the fact that $(1-\alpha)400 = 280$, we have that the expected number of point generated by this process is $400$.
\end{itemize}

\begin{figure}[tb]
    \centering
    \begin{subfigure}[b]{0.4\textwidth}
        \includegraphics[width=\textwidth]{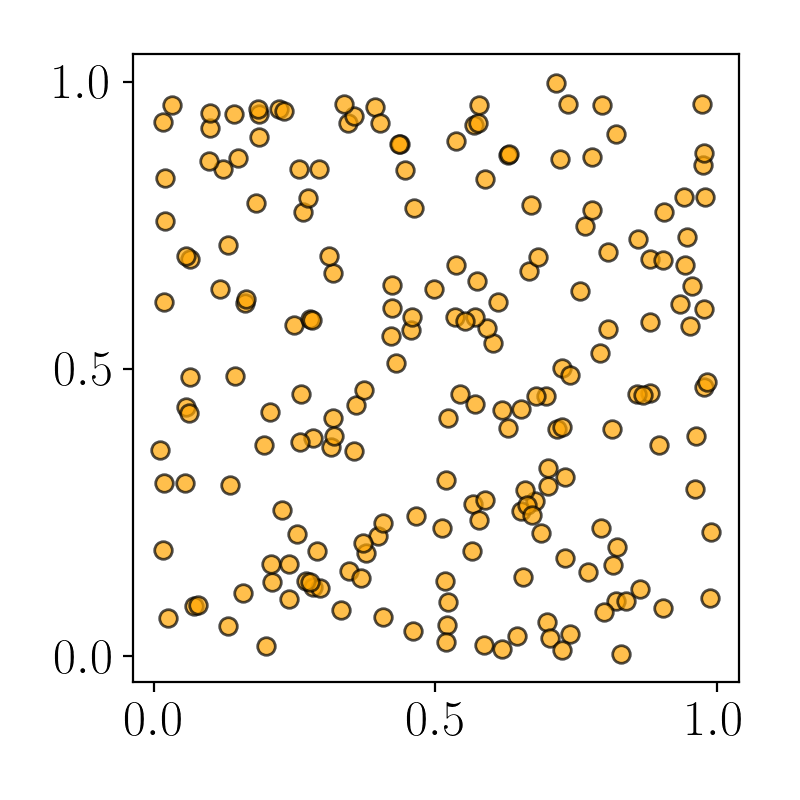}
        \caption{} 
        \label{fig:poisson}
    \end{subfigure}
    \quad
    \begin{subfigure}[b]{0.4\textwidth}
        \includegraphics[width=\textwidth]{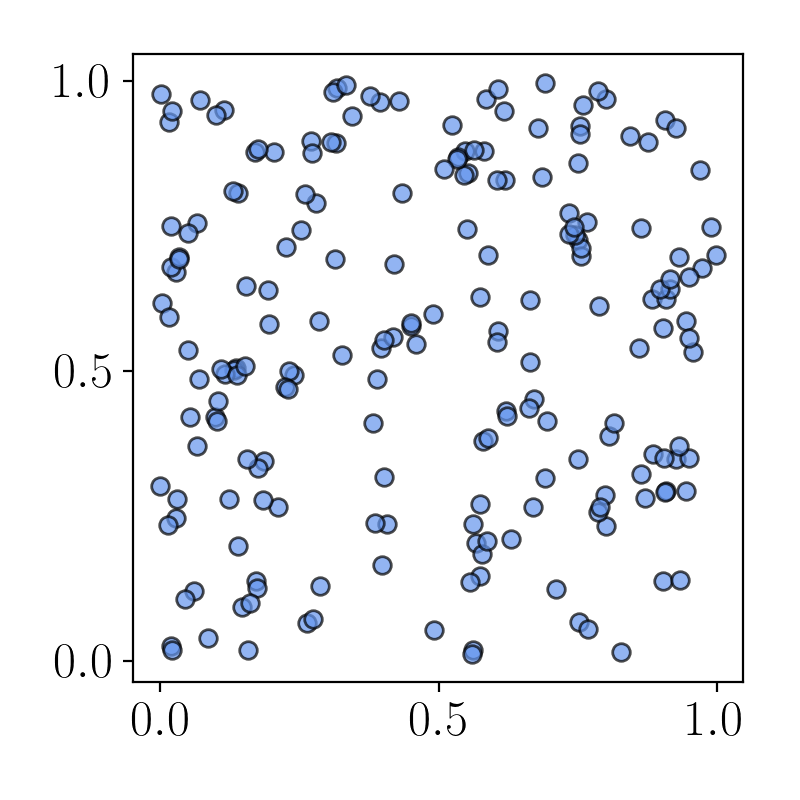}
        \caption{} 
        \label{fig:cluster}
    \end{subfigure}
    \caption{Point processes, corresponding to a Poisson point process in (a) and Hawkes cluster process in (b).}
    \label{fig:points_processes}
\end{figure}
\begin{figure}[tb]
    \centering
    \begin{subfigure}[b]{0.4\textwidth}
        \includegraphics[width=\textwidth]{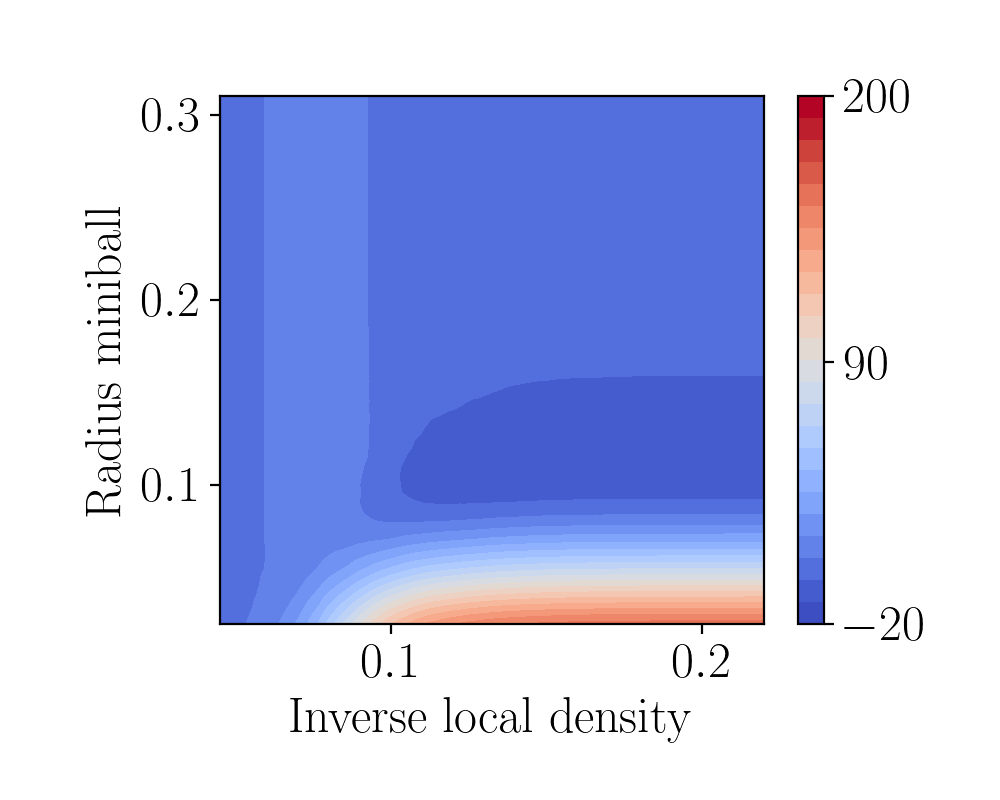}
        \caption{}
    \end{subfigure}
    \quad
    \begin{subfigure}[b]{0.4\textwidth}
        \includegraphics[width=\textwidth]{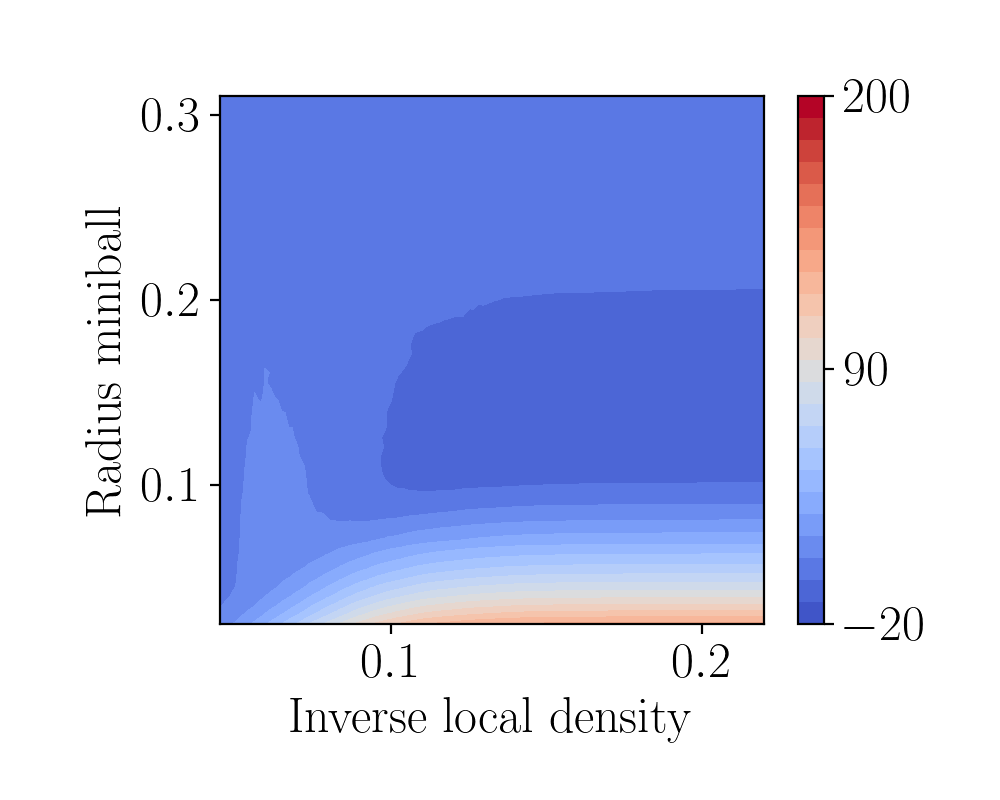}
        \caption{}
    \end{subfigure}
    
    \begin{subfigure}[b]{0.4\textwidth}
        \includegraphics[width=\textwidth]{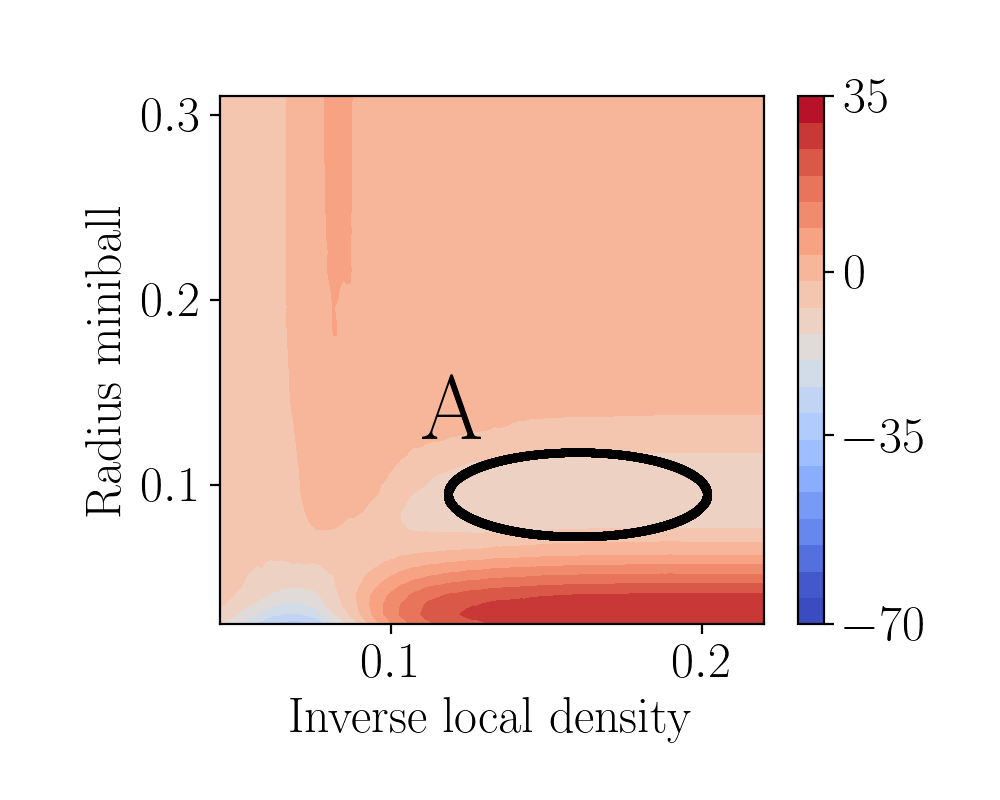}
        \caption{}
        \label{fig:points_avg_surf_diff}
    \end{subfigure}
    \quad
    \begin{subfigure}[b]{0.4\textwidth}
        \includegraphics[width=\textwidth]{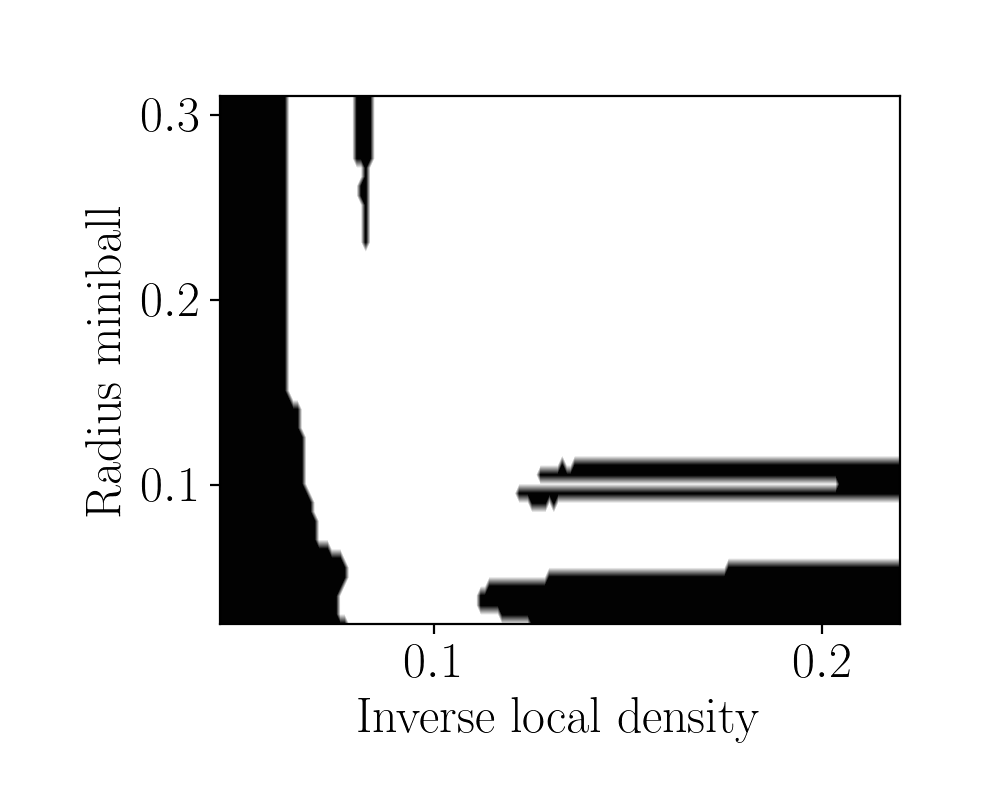}
        \caption{}
    \end{subfigure}
    \caption{Average Euler characteristic surfaces obtained from the point processes in Figure~\ref{fig:points_processes}. In (a) and (b) the average surfaces of points obtained from a Poisson and Hawkes cluster process respectively. In (c) the Euler terrain representing their difference. In (d) the black area represent regions of the parameter space where the two average surfaces in (a) and (b) are significantly different.}
    \label{fig:avg_points_surf}
\end{figure}

We generate $100$ finite point sets from both types of processes and define bi-filtering functions $\mathbf{h} = (h_1, h_2): D \rightarrow \mathbb{R}^2$ on the Delaunay triangulation $D$ of each of them setting
\begin{itemize}
    \item $h_1 = h_X$, the $\alpha$-filtration values (which is related to the distance to $D$) \cite{edelsbrunner1994three}; 
    \item $h_2(v) = \sqrt{\sum_{u\in U} \frac{\Vert v - u \Vert^2}{\vert U \vert}}$ for each vertex $v\in D$, where $U$ is the set of $k$-nearest neighbours of $v$, and $h_2(\sigma) = \max_{v\in \sigma} h_2(v)$ for each $\sigma\in D$. This filtering function estimates the inverse of the density at each vertex $v$, and extends it with the maximum to higher-dimensional simplices.
\end{itemize}
The resulting average Euler characteristic surfaces are in Figure \ref{fig:avg_points_surf}. Their absolute value difference in Figure \ref{fig:points_avg_surf_diff}.  In Figure~\ref{fig:point_process_diff}, we show the complexes at built on two instances of the cluster process and two instances of the Poisson process, at parameters where the Euler surfaces differ and the difference can clearly be seen.  We later use this same approach to understand a real data set.




\begin{figure}[!htbp]
    \centering
    \begin{subfigure}{.24\textwidth}    
        \centering  
        \hspace*{-0.9cm}\includegraphics[width=1.1\textwidth]{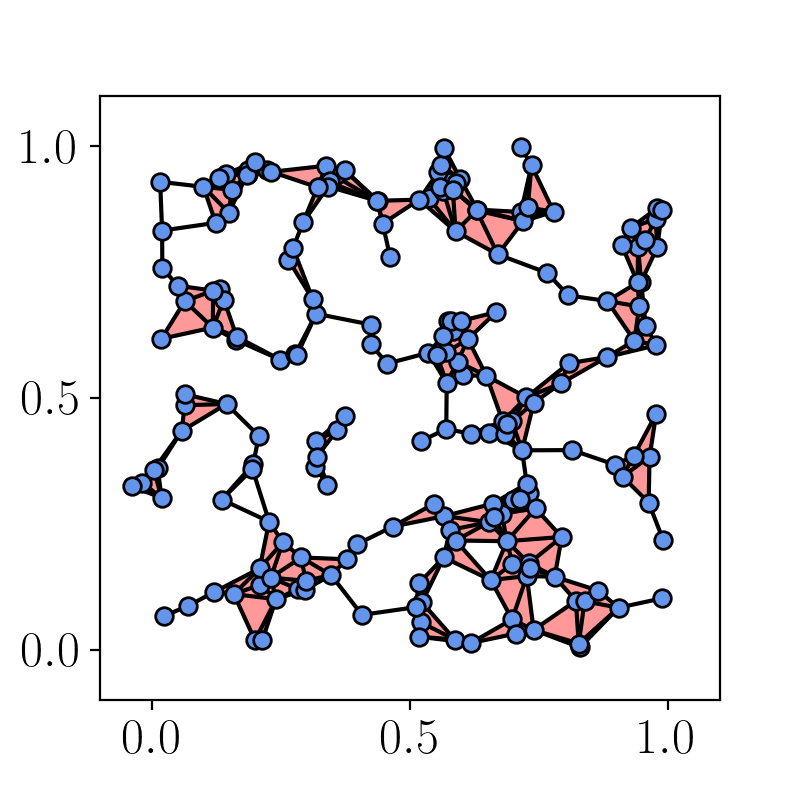}  
    \caption{}
    \end{subfigure}
    \begin{subfigure}{.24\textwidth}      
        \centering
        \hspace*{-0.5cm}\includegraphics[width=1.1\textwidth]{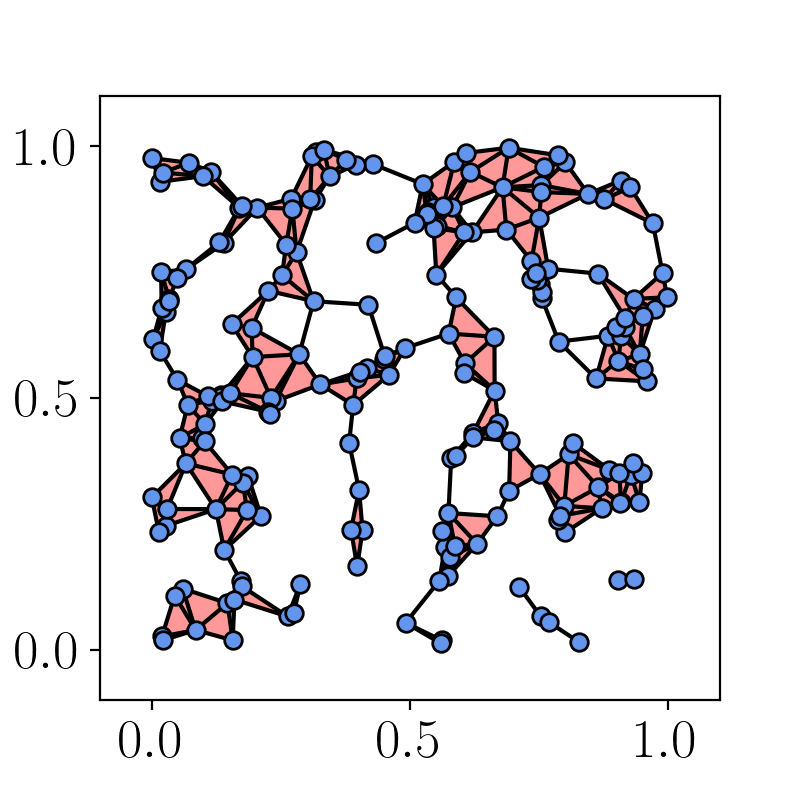}  
        \caption{}
        \end{subfigure}
    \begin{subfigure}{.24\textwidth}
        \centering \hspace*{-0.3cm} \includegraphics[width=1.1\textwidth]{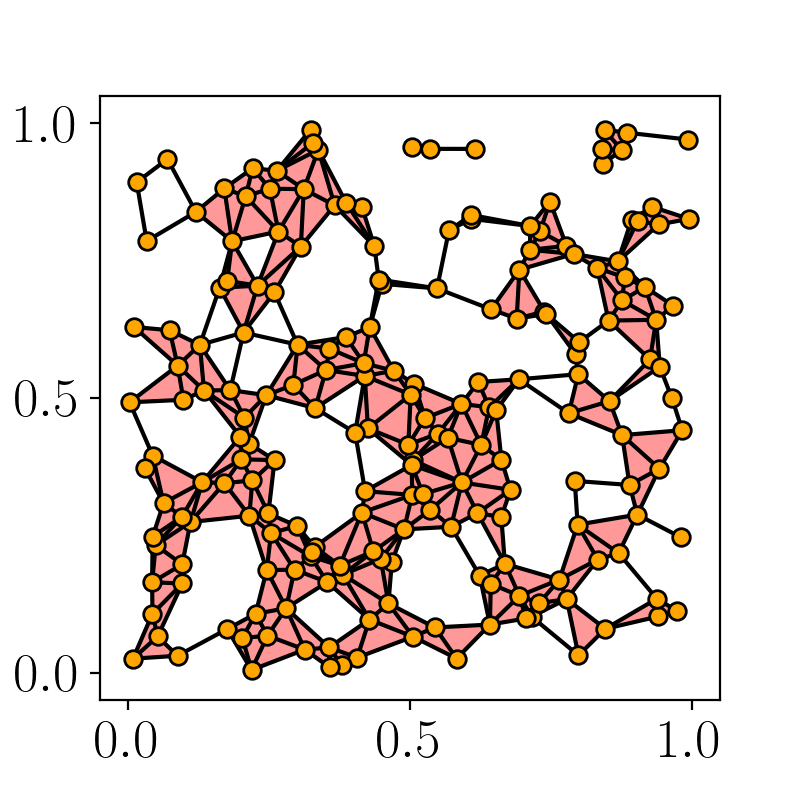}  
    \caption{}
    \end{subfigure}
    \begin{subfigure}{.24\textwidth}      
        \centering  \includegraphics[width=1.1\textwidth]{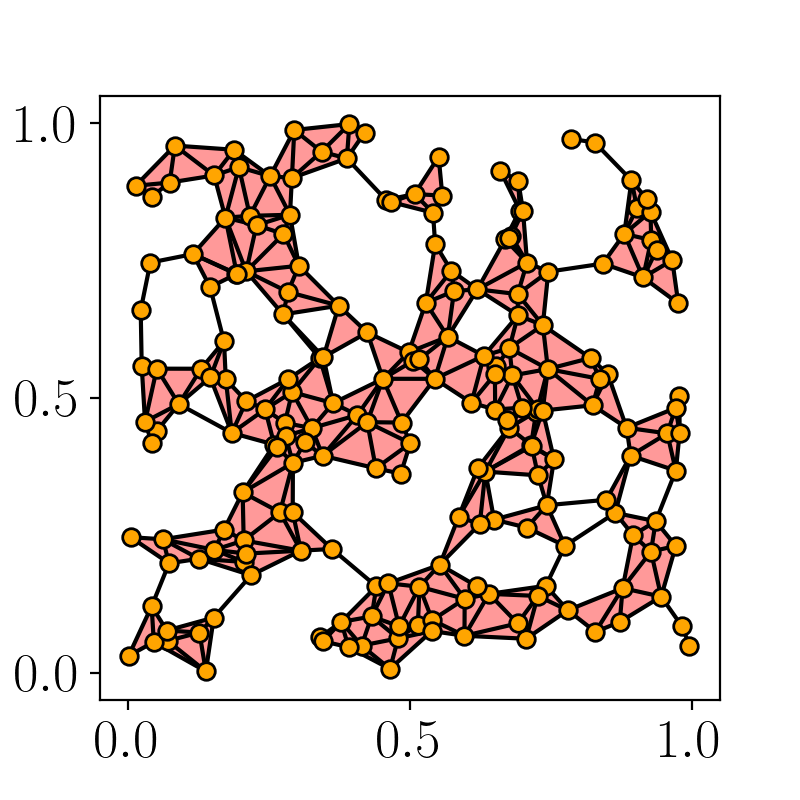}  
        \caption{}
        \end{subfigure}
    \caption{Examples of the complex built at the radius and density corresponding to Region A in Figure~\ref{fig:avg_points_surf}(c) for a cluster process (a,b) and the Poisson process (c,d)}
    \label{fig:point_process_diff}
\end{figure}

\section{Application to Detection of Diabetic Retinopathy}
\label{sec:realworld}





Here we present an application of the techniques to a real-world medical dataset on diabetic retinopathy. First, we introduce the dataset and provide evidence that the Euler characteristic curve is useful for detecting the condition using Euler characteristic curves (ECC). As the number of samples is small, the extension to surfaces is not the primary goal. Rather, having established that the Euler characteristic does capture the relevant information, we come to one of the main contributions of this paper. We use the surfaces to identify regions of parameter space which distinguish normal samples from non-normal samples. This kind of visualization is highly useful, as it is simple to then look at the underlying images and can be used to demonstrate the relevant direction/super or sub level set where difference between healthy and ill cases.  

\paragraph{Diabetic retinopathy (DR)} is a common consequence of diabetes~\cite{tey2019optical, ekoe2008epidemiology}, and a leading cause of blindness worldwide~\cite{nhs2020dr}. The disease causes a degeneration of blood vessels in the retina -- Figure~\ref{fig:comp}(a)~and~(c) shows a comparison. The signs of DR are intuitive to trained physicians, but not easy to detect computationally. We found that Euler characteristic based features are effective in detecting diabetic retinopathy on 2 different datasets (NHS Lothian and OCTAGON~\cite{diaz2019automatic}) of size $N=51$ and $N=43$ respectively. As input, we have images of the blood vessels in the eye from different patients via a method called Optical Coherence Tomography Angiography. The images can be thought of as a grayscale images with an intensity value where higher values correspond to the presence of blood vessels. In many applications, these blood vessels are first extracted as a graph or tree and then certain graph features are used as input for classification algorithms \cite{yao2020quantitative,giarratano2020framework}.  

\myparagraph{OCTA.} Optical Coherence Tomography Angiography (OCTA) has been one of the most important prospective imaging modalities in the retinal imaging domain over the last few years \cite{jia2012split, li2018quantitative}. It enables the visualisation of retinal blood vessels in a rapid and non-invasive way \cite{li2018quantitative}. Even though OCTA is not the established modality for diagnosing DR in the clinical practice, it has the potential to help doctors diagnose patients at the very early stages of DR, which are crucial for prevention and treatment. There are early signs of DR which are more likely to be observed on OCTA images before they are apparent on fundus examination \cite{thompson2019optical}. Examples of OCTA images of healthy patients (Controls), patients with DR and patients with Diabetes who have not developed DR (NoDR) are provided in Figure \ref{fig:comp}.



\myparagraph{Overview of experimental results with ECC.}
We first compare the ECC between healthy and diseased patients on 2 different OCTA datasets, coming from 2 different imaging devices (NHS Lothian (Optovue RTVue XR Avanti system (Optovue Inc., Fremont, CA)), $N=51$ and OCTAGON (DRI OCT Triton system (Topcon Corp, Tokyo, Japan)), $N=43$). The input images have resolution $304 \times 304$ and $320 \times 320$ for NHS Lothian and OCTAGON respectively. 


Using the ECC based on the intensity levels (sub and superlevel sets), we apply the method to 2 image classification tasks: 2 class classification (Control vs. DR) and 3 class classification (Control vs. NoDR vs. DR). As seen in Table \ref{fig:2_class_classification_results} and \ref{tab:control_dr_nodr}, the ECC performs better than a baseline of 2 biomarkers and comparable to state-of-the-art approaches. In particular, we achieve AUC of $0.88$ in the Control vs. DR study for NHS Lothian, $0.91$ for OCTAGON and AUC of $0.80$ for Control, $0.70$ for NoDR and $0.86$ for DR in the 3-class study (NHS Lothian). Most notably, we achieve accuracy of $81 \%$ for NoDR, which is the hardest class to classify due to the lack of well-defined signs in the image and this accuracy is better than the transfer learning approach.



\myparagraph{Overview of experimental results with ECS.} 
Due to the small sample sizes, we do not attempt classification with Euler surfaces (as the curves already perform well). Rather, we use 2-parameter ECS to look for new topological features. 
Using the same  OCTA datasets as above, we use the image intensity as our first function.  For the second function, we used:
\begin{enumerate}
    \item the complement function as illustrated in \ref{example_level_sets}, e.g. each point represents intensities in a range,
    \item the radial gradient image as seen in Figure \ref{fig:rgi}. 
\end{enumerate}
The rationale for the first one is to identify useful intervals which identify the illness, where a range of intensities approximately identify the blood vessel thicknesses, as thicker blood vessels correspond to higher intensities. The second function is to identify if there is a distance from the center of the retina image (which mostly coincides with the center of the FAZ) where blood vessels begin to behave differently. These choices are not exhaustive but demonstrate one of the main advantages to our approach -- they are readily interpreted in the context of the data, i.e. they have a clear physiological interpretation.

 To test the approach, we construct two separate pipelines (\textit{level set pipeline} and \textit{radial gradient pipeline}), depending on the function used. As seen in Figures \ref{octagon_landscape_regions}, we identify regions of interest with resulting topological biomarkers for each pipeline, for which we calculate a correlation coefficient with 2 known biomarkers. We detect that one of these biomarkers, vessel density, is strongly correlated with Pearson correlation coefficient of $0.86$ with EC from the level set pipeline. Moreover, we suggest a new topological EC biomarker which has not been reported in the literature before, moderately correlated with FAZ area with Pearson correlation coefficient of $0.55$, which is the end discovery result of the radial gradient pipeline.


\begin{figure}[!htbp]
\centering
\begin{subfigure}[b]{.3\linewidth}
\includegraphics[width=\textwidth]{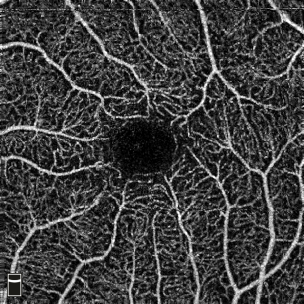}
\caption{Control: No disease}
\label{control_example_image}
\end{subfigure}%
\hfill
\begin{subfigure}[b]{.3\linewidth}
\includegraphics[width=\textwidth]{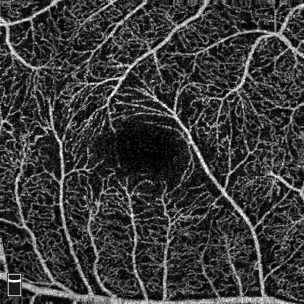}
\caption{NoDR: Diabetes, but no retinopathy.}
\end{subfigure}
\hfill
\begin{subfigure}[b]{.3\linewidth}
\includegraphics[width=\textwidth]{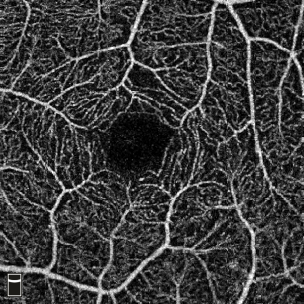}
\caption{DR: Diabetes with retinopathy.}
\label{dr_example_image}
\end{subfigure}
\caption{OCTA scans from Control, DR and NoDR patients. Changes to the microvasculature are apparent with disease progression. For example, the vessel density reduces and the foveal avascular zone (FAZ), which is the black regions in the middles of the image, is enlarged and distorted with less circular shape.} \label{fig:comp}
\end{figure}

\subsection{ECC classification study}

We provide the first study of the global topological structure of the OCTA retinal images via the ECC. As mentioened above, we take the filtering function to be the pixel intensity value. We calculate the EC for all the possible pixel threshold values, which in the case of gray-scale images are 256. The ECC is then used as a feature vector  for classification.

\myparagraph{2-class classification (Control vs. DR)}W e use for comparison a baseline of biomaker measurements for two known biomarkers associated with disease progression, vessel density (VD) and FAZ area, similar to the study in \cite{sandhu2018automated}, where they use 3 biomarkers (VD, FAZ area and vessel calibre). Moreover, we further compare it with the state-of-the-art deep learning approaches to patient classification. A VGG16 architecture with transfer learning was used as described in \cite{andreeva2020dr} to classify the same OCTA images.

\begin{table}[!htbp]
    \begin{minipage}{.5\linewidth}
    \label{tab:nhslothian_control_dr}
      \centering
        \begin{tabular}{c  c  c  c}
            \hline
            & \multicolumn{2}{ c }{\textbf{NHS Lothian, Control vs. DR}} \\ 
            \cline{1-4}
            & \textbf{Baseline} & \textbf{Our approach} &   \textbf{VGG16}\\
            Overall Acc  & $0.72 \pm 0.03$ &   $0.81 \pm 0.04$ &   \textbf{0.84 $\pm$ 0.07}\\
            Sen (Control) & \textbf{0.94 $\pm$ 0.05} &  \textbf{0.94 $\pm$ 0.05} &  $0.88 \pm 0.07$\\
            Spe (Control) & $0.30 \pm 0.06$ &  $0.60 \pm 0.06$  &  \textbf{0.77 $\pm$ 0.09}\\
            AUC  & $0.75 \pm 0.06$ &  \textbf{0.88 $\pm$ 0.03} &  \textbf{0.88 $\pm$ 0.12}\\ \hline 
        \end{tabular}
    \end{minipage}%
    \begin{minipage}{.5\linewidth}
      \centering
        \label{tab:octagon_control}
        \begin{tabular}{c  c  c  c}
            \hline
            & \multicolumn{2}{ c }{\textbf{OCTAGON, Control vs. DR}} \\ 
            \cline{1-4}
            & \textbf{Baseline} & \textbf{Our approach} &   \textbf{VGG16}\\
            & $0.82 \pm 0.04$ &   \textbf{0.87 $\pm$ 0.04} &   $0.84 \pm 0.07$\\
            & $0.87 \pm 0.04$ &  0.96 $\pm$ 0.04 &  \textbf{1.00 $\pm$ 0.00}\\
            & \textbf{0.71 $\pm$ 0.08} &  \textbf{0.71 $\pm$ 0.08} &  $0.53 \pm 0.20$\\
            & $0.87 \pm 0.04$ &  0.91 $\pm$ 0.03 &  \textbf{0.94 $\pm$ 0.06}\\ \hline 
        \end{tabular}
    \end{minipage} 
        \caption{Table of classification performances in the Control vs. DR study}
        \label{fig:2_class_classification_results}
\end{table}

In Table \ref{fig:2_class_classification_results} we can see the classification performance in the 2-class task (Control vs. DR). Our approach performs better than the baseline in all metrics, but the VGG16 method with data augmentation outperforms for accuracy and specificity. The AUC in both cases is comparable and it has higher variance for VGG16. For OCTAGON, we observe that we achieve a 5\% accuracy improvement over the baseline and 3\% improvement over the transfer learning approach for OCTAGON. For the other metrics, the ECC approach is still better than the baseline, and the results are at worst comparable with VGG16.

\myparagraph{3-class classification (Control vs. NoDR vs DR).} We compared the results with the transfer learning approach applied to the same dataset with data augmentation. The classification statistics with data augmentation for VGG 16 are displayed in Table \ref{tab:control_dr_nodr}. 

\begin{table}[H]
\begin{center}
\resizebox{\columnwidth}{!}{
\begin{tabular}{c  c  c  c  c  c  c }
\hline
 & \multicolumn{2}{ c }{\textbf{Controls}} & \multicolumn{2}{ c }{\textbf{NoDR}} &  \multicolumn{2}{ c }{\textbf{DR}} \\ 
 \cline{1-7}
  & \textbf{Our approach} & \textbf{VGG16} & \textbf{Our approach} & \textbf{VGG16} & \textbf{Our approach} & \textbf{VGG16}   \\
  \textbf{ACC}  &   0.68 $\pm$ 0.05 & \textbf{0.78 $\pm$ 0.05} &  \textbf{0.81 $\pm$ 0.04} &  $0.72 \pm 0.04$ & \textbf{0.76 $\pm$ 0.02} & \textbf{0.77 $\pm$ 0.04} \\
  \textbf{SEN} &   $0.79 \pm 0.05$ & \textbf{0.90 $\pm$ 0.05} &  \textbf{0.56 $\pm$ 0.06} & $0.20 \pm 0.13$ &   0.26 $\pm$ 0.14 &  \textbf{0.55 $\pm$ 0.11} \\
  \textbf{SPE} &  0.57 $\pm$ 0.09 & \textbf{0.67 $\pm$ 0.11} &  \textbf{0.90 $\pm$ 0.04} &  0.88$\pm$ 0.05 & \textbf{0.90 $\pm$ 0.02} &  $0.86 \pm 0.05$    \\ 
  \textbf{AUC} & 0.80 $\pm$ 0.04 & \textbf{0.90 $\pm$ 0.15} & \textbf{0.70 $\pm$ 0.04} &  0.67$\pm$ 0.28 & \textbf{0.86 $\pm$ 0.06} & $0.75 \pm 0.22$    \\
  \hline
\end{tabular}}
\end{center}
\caption{Table of classification performances with transfer learning}
\label{tab:control_dr_nodr}
\end{table}

Our results are comparable to the results with data augmentation. Our approach demonstrates good performance for the NoDR class with the highest accuracy of $0.81$, followed by DR and then Control. This result signifies its suitability for early detection tool, as the NoDR images have the slightest of changes compared to DR.

\subsection{Biomarkers}

The predominant approaches for image analysis in the OCTA literature have centered around a small number of explainable candidate biomarkers as suggested by the clinical knowledge acquired \cite{ylenia2020network}. Indeed, quantifiable features can be extracted from the OCTA images which are important biomarkers for DR. Statistical studies have identified the usefulness of local and global metrics based on the morphology of the foveal avascular zone (FAZ) and vascular-based metrics as biomarkers for distinguishing between healthy and DR eyes. Examples of the former include FAZ area, FAZ contour irregularity \cite{khadamy2018update, takase2015enlargement}, while examples of the latter are vessel caliber (VC), fractal dimension (FD), tortuosity, vessel density (VD) and geometric features of the vascular network \cite{giarratano2020framework, alam2020quantitative, le2019fully, alam2019quantitative, sasongko2011retinal}. For a more detailed review, please refer to \cite{yao2020quantitative}. However, there can be up to 25\% differences in the measurements of one of the biomarkers known to be linked to DR (vessel density) and 24\% in foveal avascular zone (FAZ) area \cite{freiberg2016optical}, which is also an early biomarker for DR and is enlarged for Diabetic patients.



\myparagraph{Vessel Density (VD).}VD is the ratio of the parts of image which are taken by blood vessels to the entire image. VD measurements were obtained as described in \cite{giarratano2020automated}. For OCTAGON, OOF filter was used, while for NHS Lothian a U-Net approach was adopted due to the availability of manually labelled data. 



\myparagraph{FAZ area.}The FAZ area is measured by segmenting the FAZ (the black region in the middle as seen in Figure \ref{control_example_image}) and calculating the total area of the segmented region. As there exist different methods for segmenting the FAZ area and depending on the dataset and the availability of manually segmented data available, the FAZ area is calculated using two different methods. The first one is used for OCTAGON and uses the FAZ segmentation and area calculation as described in \cite{diaz2019automatic}. The second one is used for NHS Lothian and follows the methodology as in \cite{giarratano2020automated}. 


\subsection{Euler characteristic surface (ECS) for identifying biomarkers}


In order to identifying a particular area in the image to look at, we would develop a suitable second image for the ECS guided by the accumulated biomedical knowledge of biomarkers VD and FAZ. We would like to focus on either:
\begin{itemize}
    \item a pixel interval, which consists of vessels with pixel values in particular range; or
    \item restrict our attention to the FAZ area in the middle of the image.
\end{itemize}

For these two particular purposes, we will use the ECS with a carefully selected second image $M_2$. As we have seen previously, this representation is richer. 




\myparagraph{Level set pipeline.} The \textit{complement image} is the image in which each pixel value is subtracted from the maximum pixel value, 256 in the case of a gray-scale image. In Figure \ref{fig:ecs_complement} 
 we can see an example of the resulting ECS from taking a Control image $M_1$ in Figure \ref{fig:left_contol_original} and its complement to be $M_2$ as shown in Figure \ref{fig:middle_complement}.


\begin{figure}[!htbp]
    \centering
 
\begin{subfigure}[b]{.32\textwidth}      
    \includegraphics[width=\textwidth]{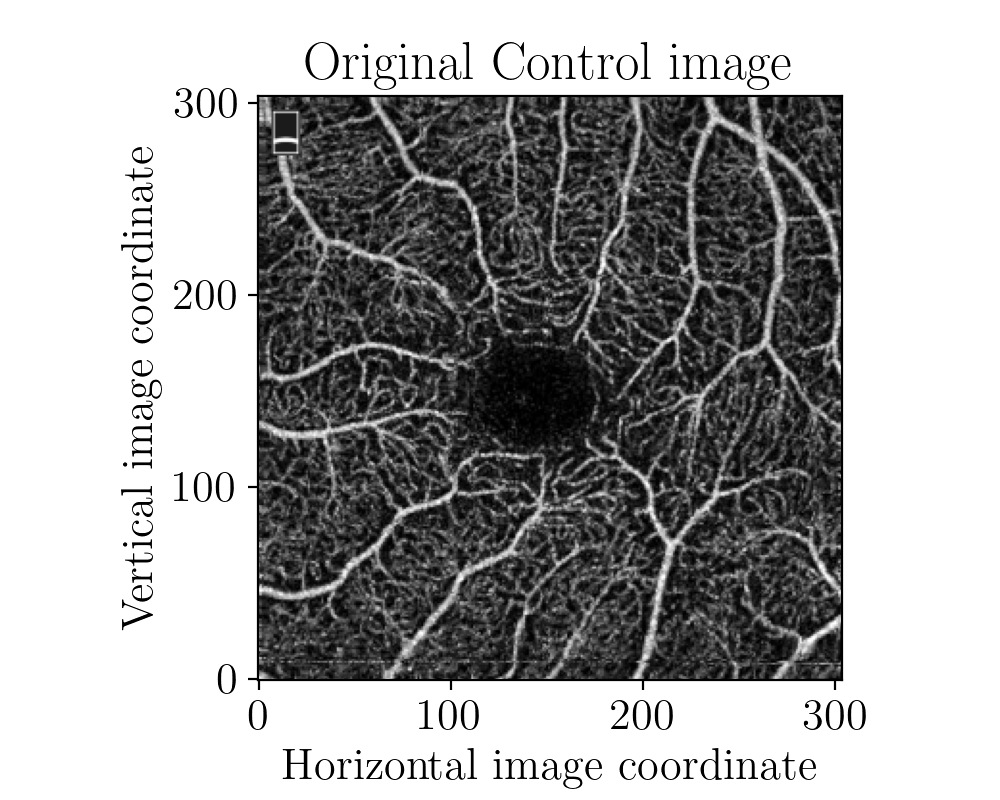}%
    \caption{Control image ($M_1$)}
    \label{fig:left_contol_original}%
\end{subfigure}
\hfill
\begin{subfigure}[b]{.32\textwidth}      
    \includegraphics[width=\textwidth]{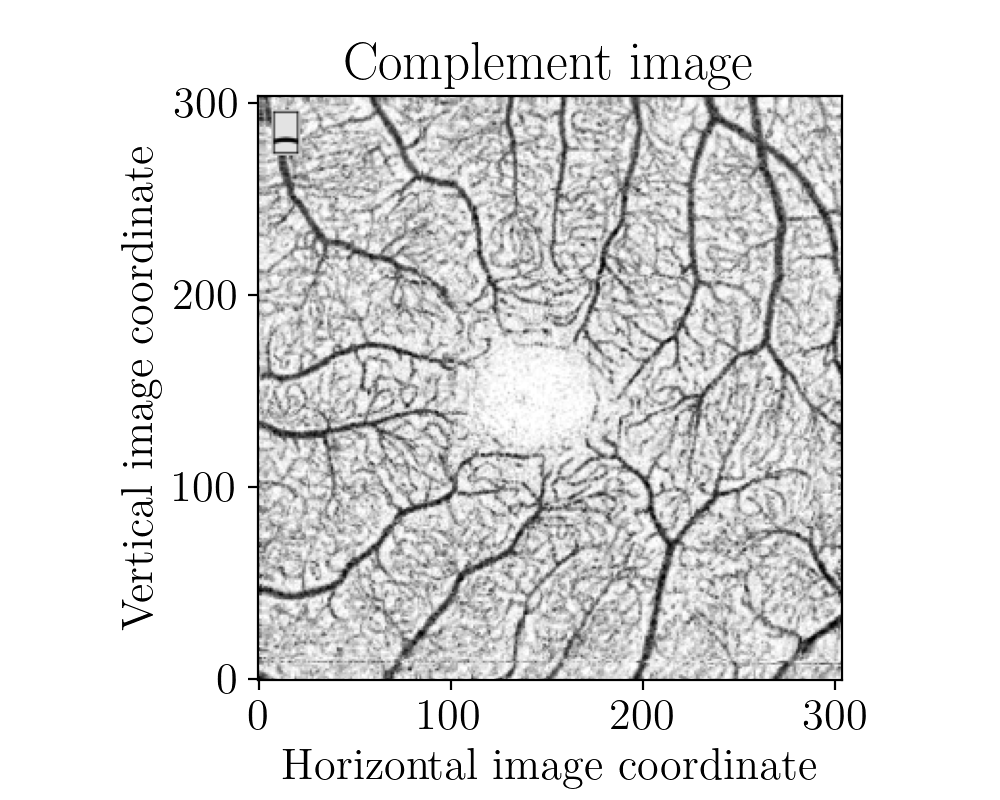}%
    \caption{Complement of control image ($M_2$)}
    \label{fig:middle_complement}%
\end{subfigure}
\hfill
\begin{subfigure}[b]{.32\textwidth}        
             \includegraphics[width=\textwidth]{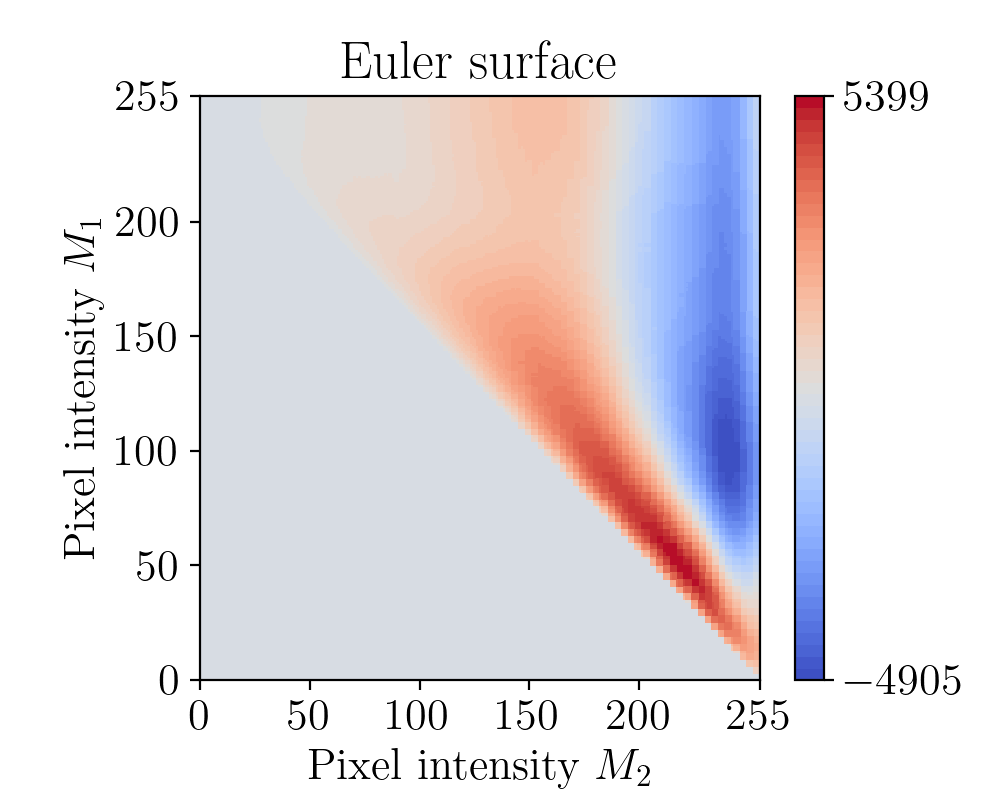}%
             \caption{ECS($M_1,M_2$)}
             \label{fig:ecs_complement}%
            \end{subfigure}
 
             \caption{Example of the ECS with $M_2$ being the complement image}
          \label{example_level_sets}
\end{figure}




\myparagraph{Radial gradient pipeline.}In Figure \ref{fig:ecs_radial_control}. 
 we can see an example of the resulting ECS from taking a Control image $M_1$ in Figure \ref{fig:ecs_control_rgi} and the second image $M_2$ in Figure \ref{fig:rgi} is the \textit{radial gradient image}. In this case, the choice of image $M_2$ is motivated by the idea to capture the FAZ in the OCTA images, as its enlargement is characteristic for the progression of DR \cite{freiberg2016optical}. By selecting a threshold $t$ for the radial gradient image, we consider the disk with radius $t$ and its intersection with the OCTA image. Thus, it is in theory possible to detect the FAZ, rendering the radial gradient image a suitable candidate for a second image $M_2$.


\begin{figure}[!htbp]
    \centering
 
\begin{subfigure}[b]{.32\textwidth}      
    \includegraphics[width=\textwidth]{images_OCTA/original_image_control_correct.png}
    \caption{Control image ($M_1$)}
    \label{fig:ecs_control_rgi}%
\end{subfigure}
\hfill
\begin{subfigure}[b]{.32\textwidth}      
    \includegraphics[width=\textwidth]{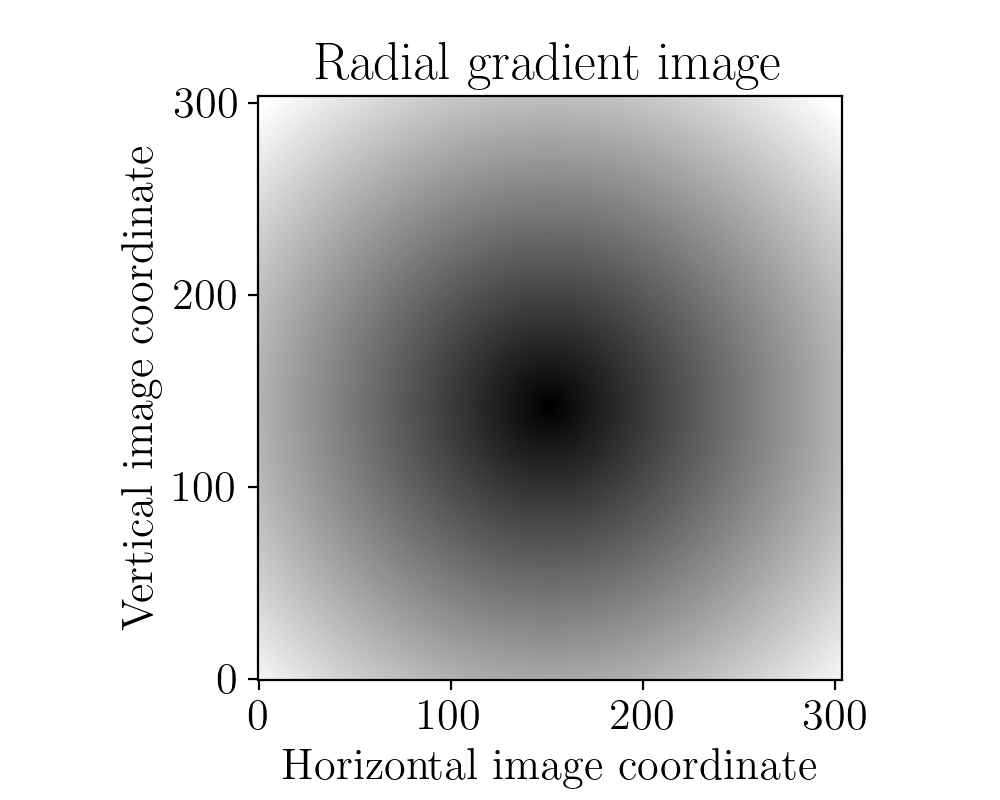}
    \caption{Radial gradient image ($M_2$)}
      \label{fig:rgi}%
\end{subfigure}
\hfill
\begin{subfigure}[b]{.32\textwidth}        
             \includegraphics[width=\textwidth]{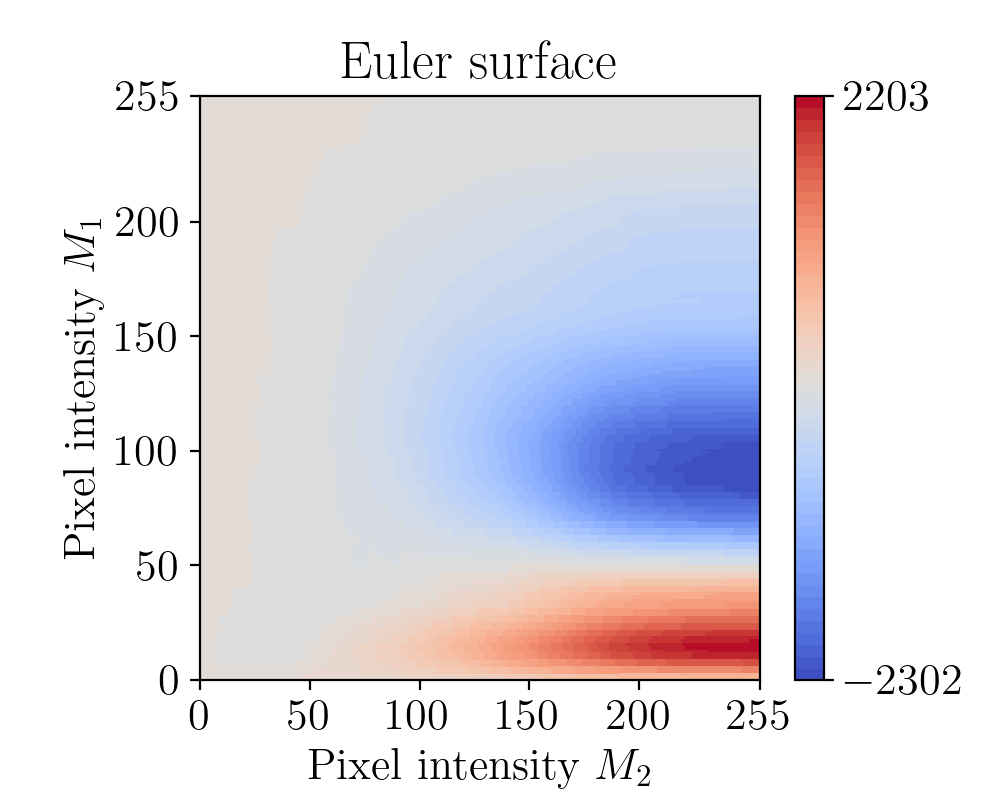}
             \caption{$\overline{\mathrm{ECS}}(M_1,M_2)$}
             \label{fig:ecs_radial_control}%
            \end{subfigure}
          \caption{Example of the ECS with $M_2$ being the radial gradient image}
          \label{fig:ecs_example}
\end{figure}



\myparagraph{Terrain.}The main insight is that there are areas on the ECS which are highly discriminatory between the two groups (Controls and DR). We compute the normalized terrains using all the images. Recall that this is constructed by comparing point-wise means and standard deviations of the EC. For OCTAGON, we identify 2 regions of interest in the level set \foo substitute in Figure \ref{fig:labelled_landscape_level} - regions A and B. For the radial gradient \foo substitute in Figure \ref{fig:labelled_landscape_radial} there is one region A which has a high density of red points. The \foo substitutes for NHS Lothian are similar, suggesting the existence of underlying and fundamental topological features, shared between datasets and well captured by the ECS.

In Figures~\ref{fig:octagon_examples_levelset} we see some examples of the images from the control and DR from Region B, i.e. the the thresholds corresponding to a point in Region B. The difference between the control and DR is visually clear which shows how the Euler surfaces can yield insights into the differences between classes. In  Figure~\ref{fig:octagon_examples_radial}, we show images for the sublevel set and radial filtration function. Here the difference is less obvious, however, the difference indicates that the shape of the FAZ (the circular region in the center which is devoid of blood vessels) becomes less circular and enlarged for DR.

\paragraph{Correlate \foo substitute regions with biomarkers.} To establish a link between the topological regions and the known biomarkers, we calculate the correlation between the EC of region B in the level set \foo substitute in Figure \ref{fig:labelled_landscape_level} and vessel density (VD) and the correlation between the EC in region A in the radial gradient \foo substitute in Figure \ref{fig:labelled_landscape_radial} and FAZ area. 

\begin{figure}[!htbp]
    \centering
    \begin{subfigure}[t]{.45\textwidth}      
    \includegraphics[width=\textwidth]{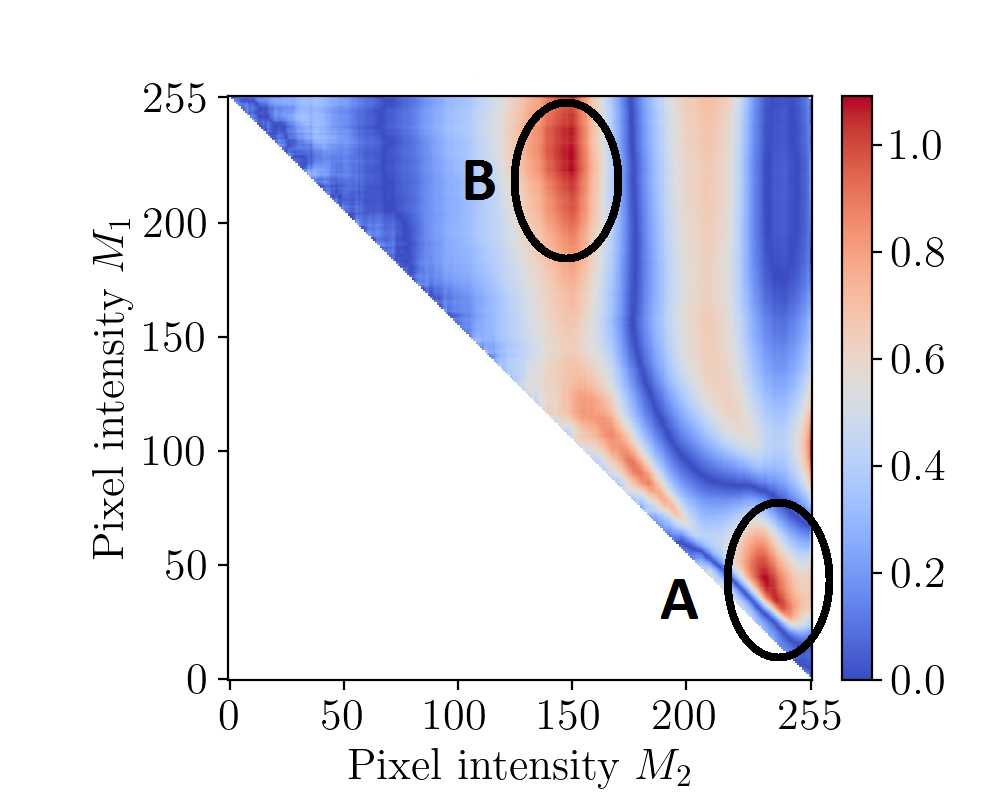}  
    \caption{Level set \foo substitute with regions A and B}
    \label{fig:labelled_landscape_level}
    \end{subfigure}
    \quad
    \begin{subfigure}[t]{0.45\textwidth}
    \includegraphics[width=\textwidth]{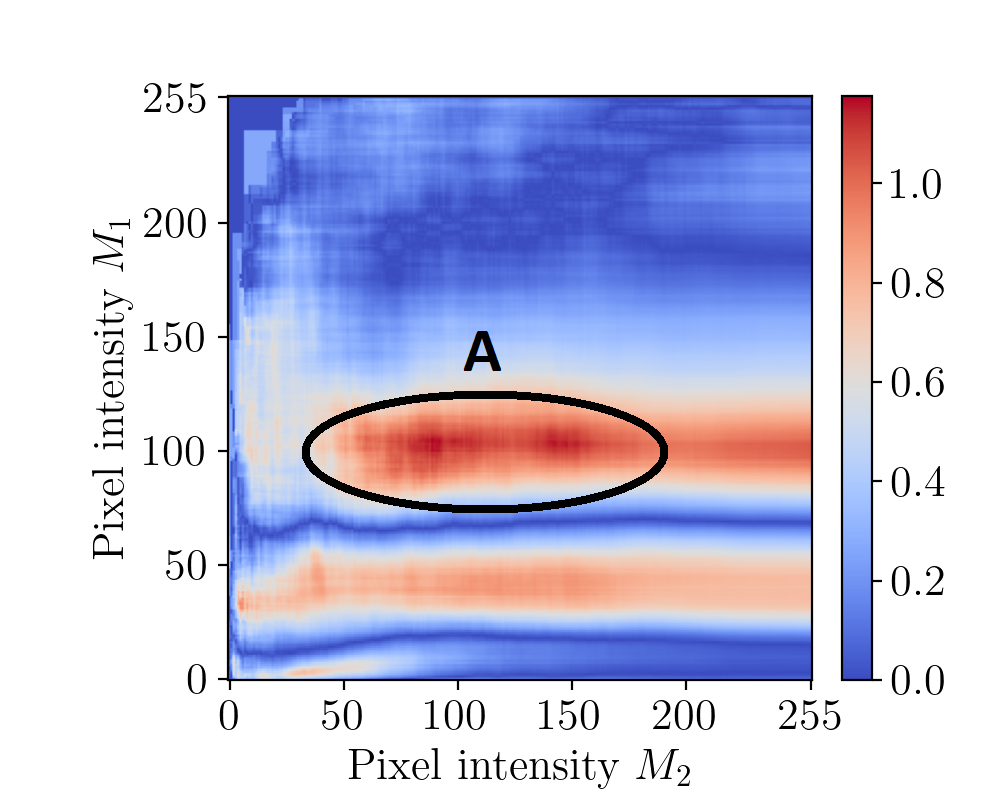}%
    \caption{Radial gradient \foo substitute with region of interest A}
    \label{fig:labelled_landscape_radial}
    \end{subfigure}
    \caption{Examples of regions of interest in the \foo substitutes of OCTAGON}
    \label{octagon_landscape_regions}
\end{figure}

\begin{table}[!htbp]
\begin{center}
\scalebox{0.9}{
\small\addtolength{\tabcolsep}{-1pt}
\begin{tabular}{c  c  c  c  c}
\hline
 & {\textbf{EC ($p$-value), NHS Lothian}} & {\textbf{EC ($p$-value), OCTAGON}}  \\ 

  \textbf{(VD,EC(VD))} &   $0.86 (1.65 \times 10^{-12})$ &  $ 0.20 (0.20) $  \\
  \textbf{(FAZ,EC(FAZ))} &   $0.55 (2.52 \times 10^{-4})$ & $0.57 (7.13 \times 10^{-5})$ \\
  \hline
\end{tabular}}
\caption{Correlation results for NHS Lothian and OCTAGON}
\label{tab:correlation_both_datasets}
\end{center}
\end{table}

\myparagraph{Correlation results.} We can see a summary of the correlation results in Table \ref{tab:correlation_both_datasets}. For FAZ area, the correlation results for NHS Lothian are consistent with the results for OCTAGON. Therefore, there is preliminary evidence that the method used is robust between datasets from a different device and that correlation between EC(FAZ) and FAZ area is maintained at the similar levels across devices. However, this is not the case for VD. For OCTAGON there is not significant correlation. This could be attributed to the fact that NHS Lothian on average has higher vessel valid visibility and less motion artifacts \cite{li2018quantitative}.

\begin{figure}[!htbp]
    \centering
    \begin{subfigure}{.24\textwidth}    
        \centering  
        \hspace*{-0.9cm}\includegraphics[width=1.4\textwidth]{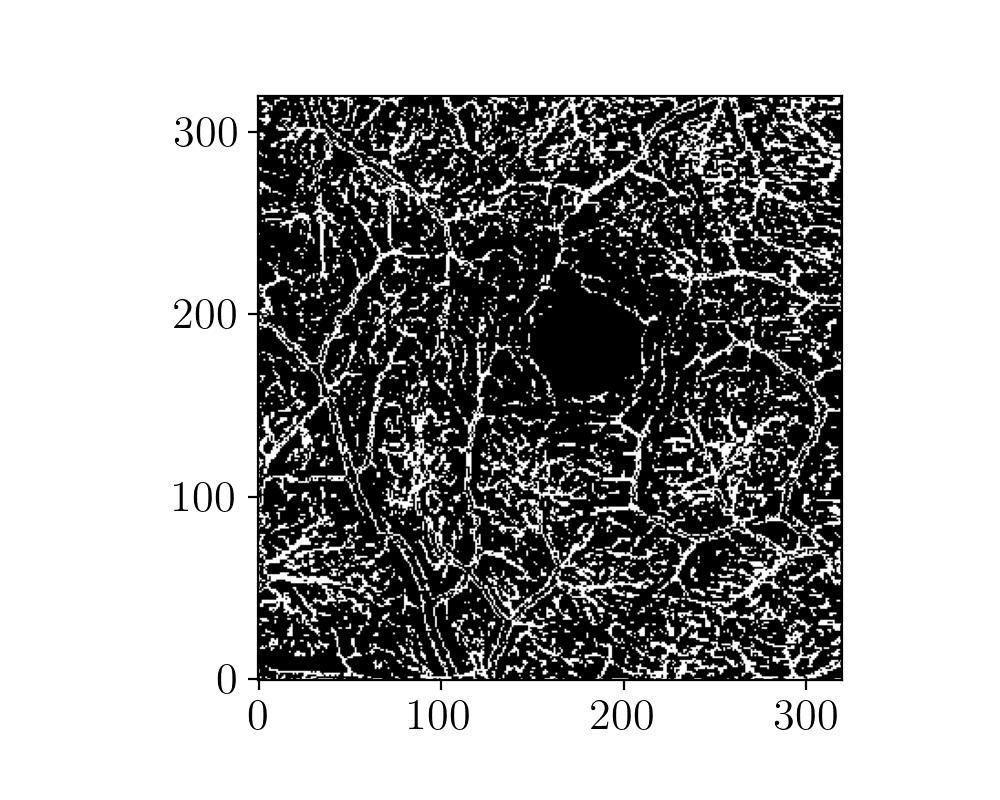}  
    \caption{}
    \end{subfigure}
    \begin{subfigure}{.24\textwidth}
        \centering \hspace*{-0.9cm}\includegraphics[width=1.4\textwidth]{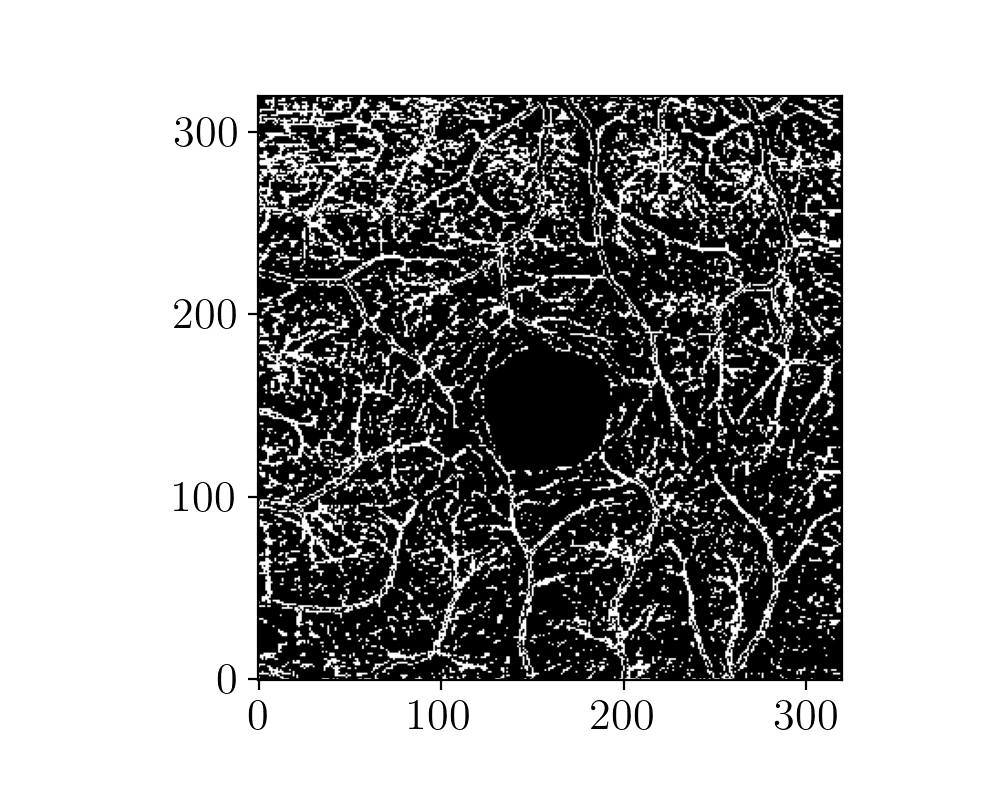}  
    \caption{}
    \end{subfigure}
    \begin{subfigure}{.24\textwidth}      
        \centering
        \hspace*{-0.9cm}\includegraphics[width=1.4\textwidth]{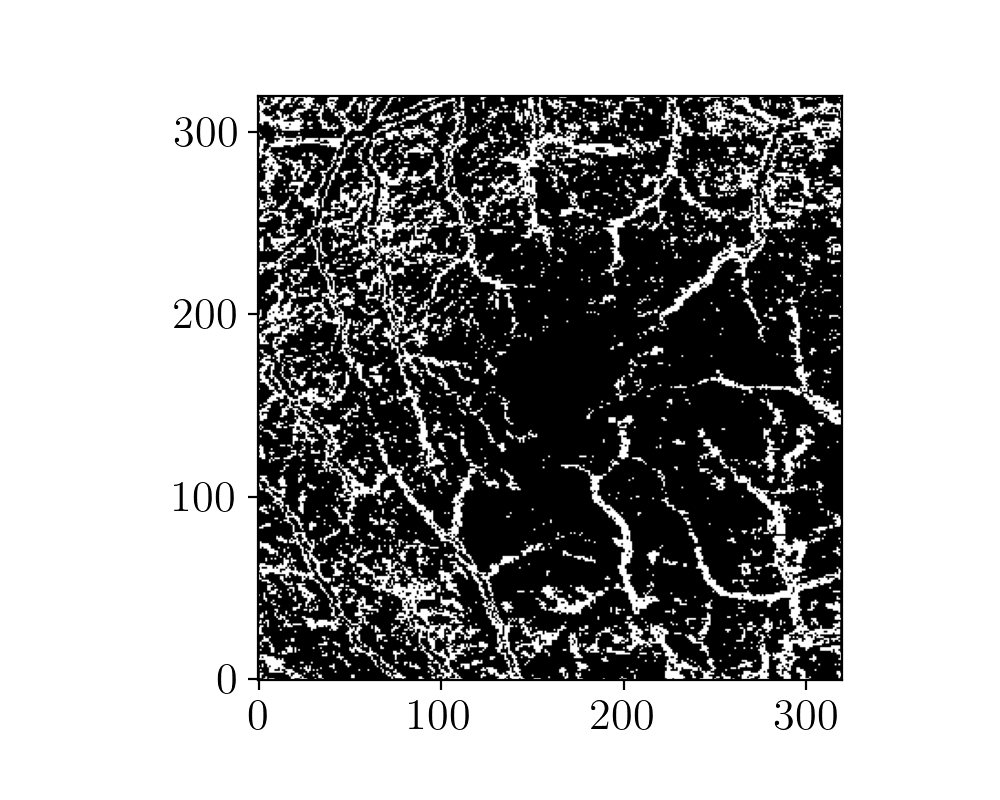}  
        \caption{}
        \end{subfigure}
    \begin{subfigure}{.24\textwidth}      
        \centering  \hspace*{-0.9cm}\includegraphics[width=1.4\textwidth]{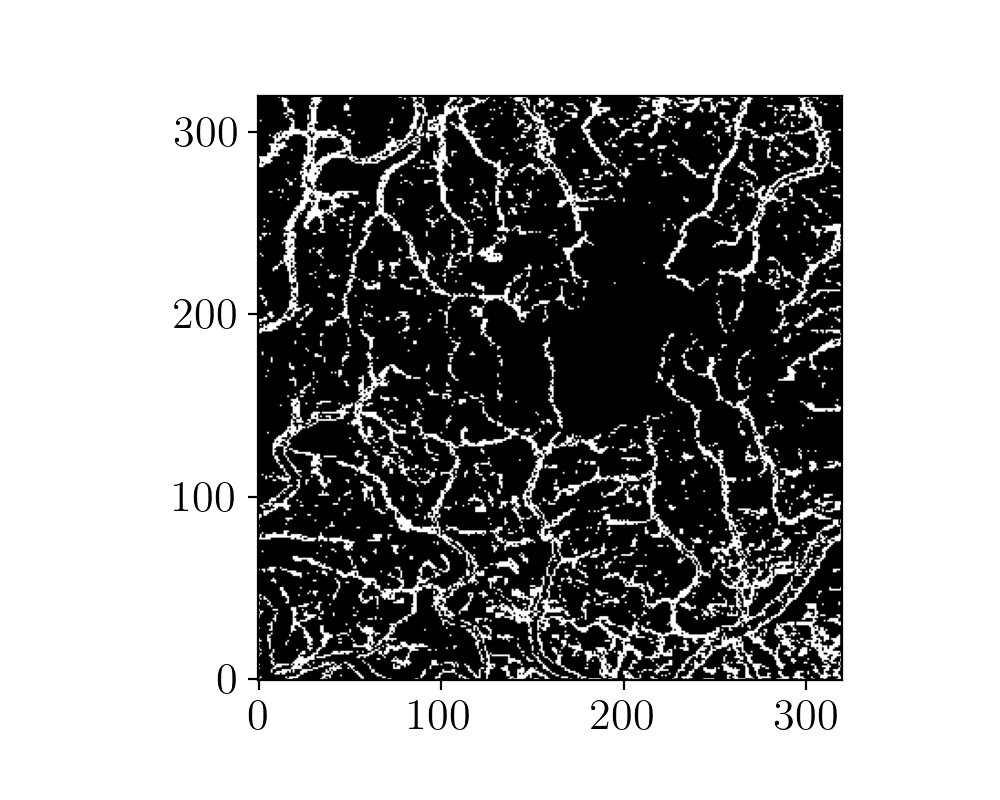}  
        \caption{}
        \end{subfigure}
    \caption{Examples of level-sets of images of the control (a,b) and  DR (c,d) corresponding to Region B.}
    \label{fig:octagon_examples_levelset}
\end{figure}

\begin{figure}[!htbp]
    \centering
    \begin{subfigure}{.24\textwidth}    
        \centering  
        \hspace*{-0.9cm}\includegraphics[width=1.4\textwidth]{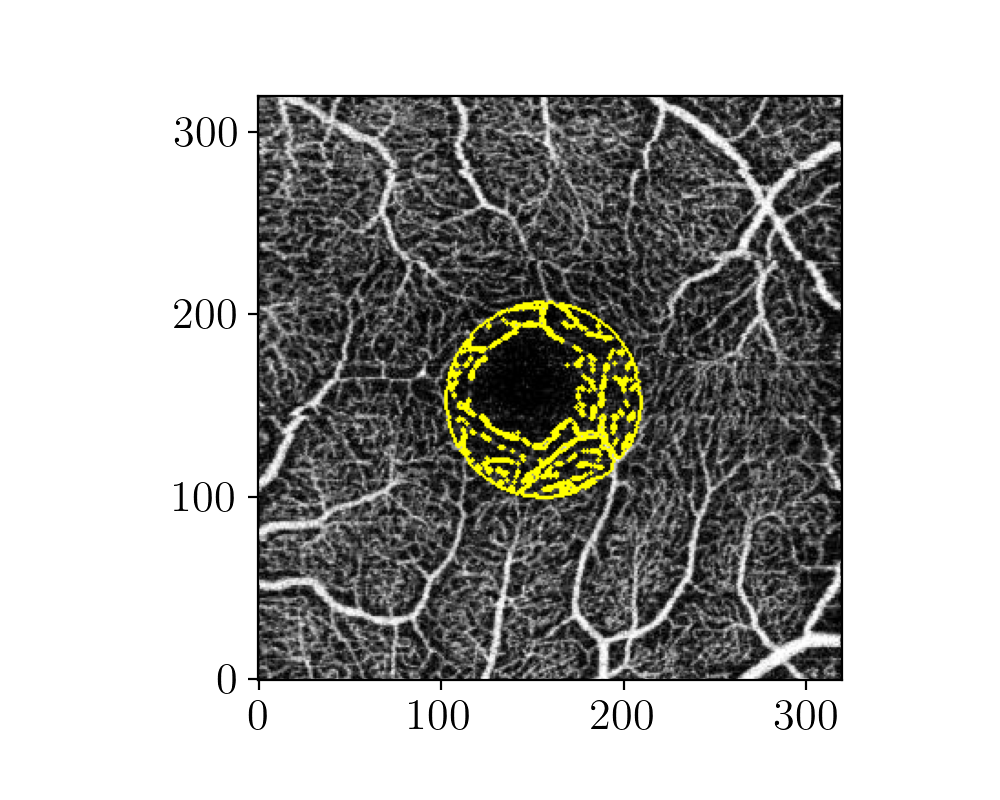}  
    \caption{}
    \end{subfigure}
    \begin{subfigure}{.24\textwidth}
        \centering \hspace*{-0.9cm}\includegraphics[width=1.4\textwidth]{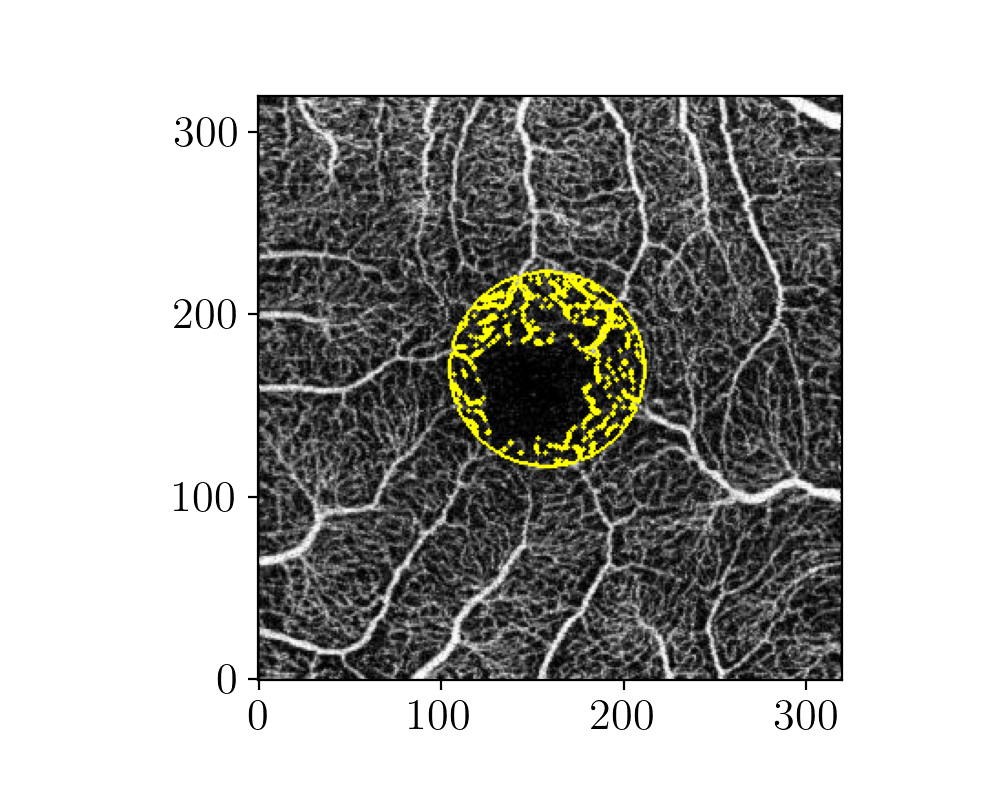}  
    \caption{}
    \end{subfigure}
    \begin{subfigure}{.24\textwidth}      
        \centering
        \hspace*{-0.9cm}\includegraphics[width=1.4\textwidth]{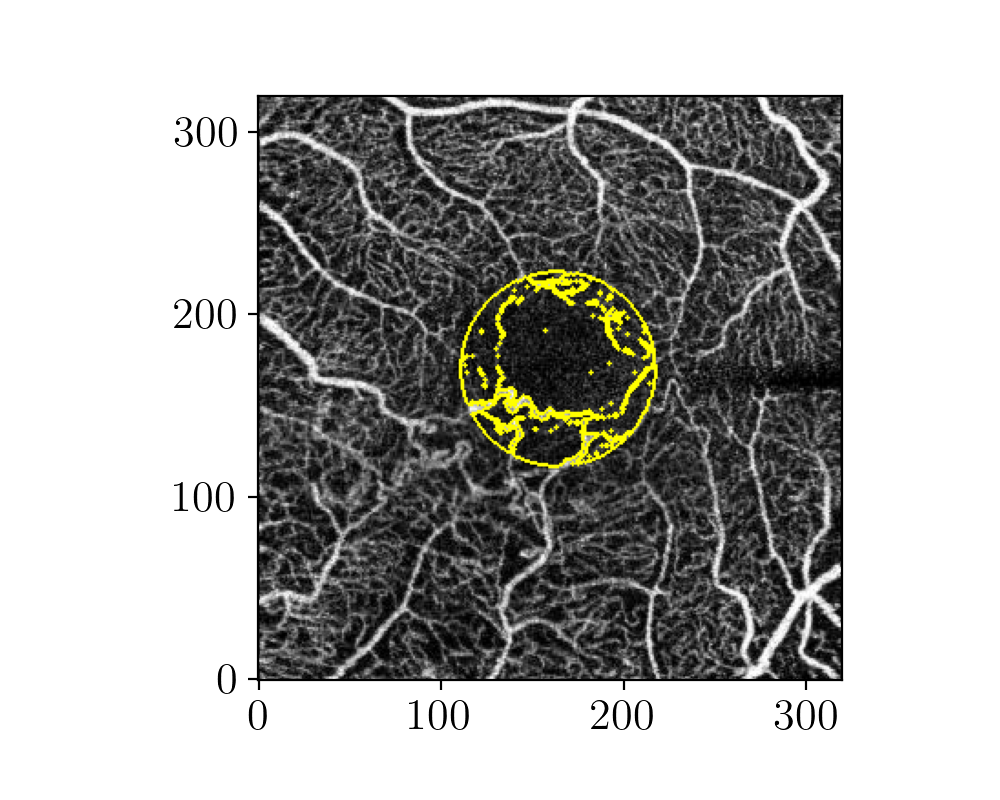}  
        \caption{}
        \end{subfigure}
       \begin{subfigure}{.24\textwidth}      
        \centering  \hspace*{-0.9cm}\includegraphics[width=1.4\textwidth]{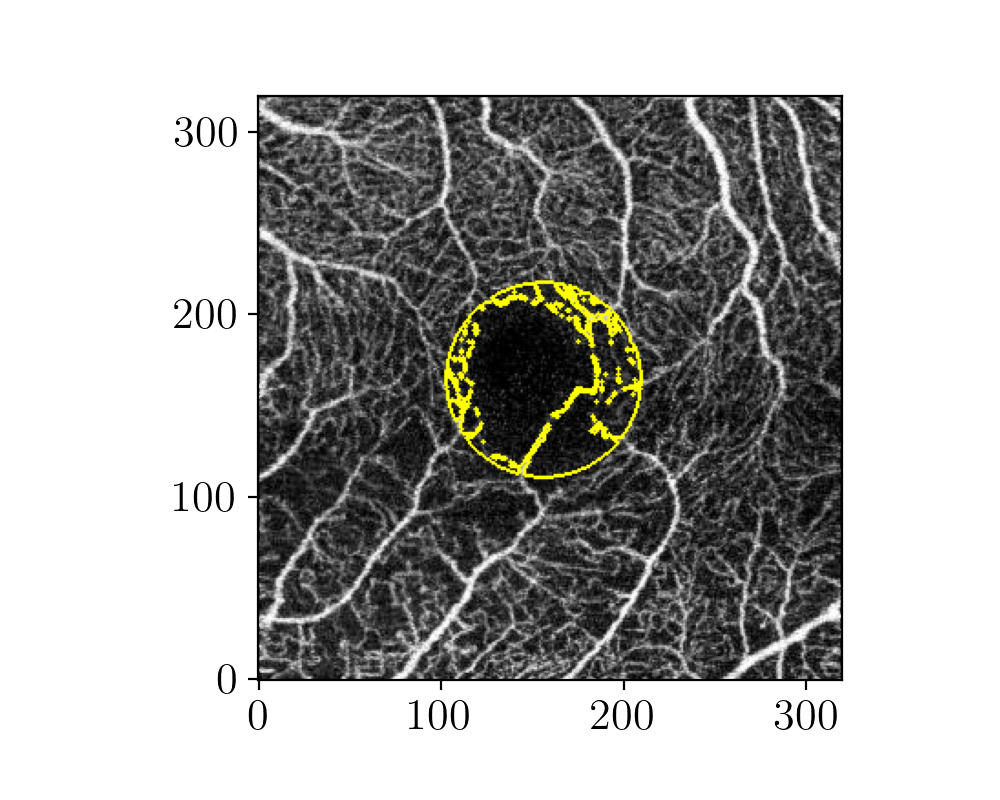}  
        \caption{}
        \end{subfigure}
    \caption{Examples of sublevel sets with the radial mask shown in yellow with the control shown in (a,b) and DR (c,d) corresponding to Region A.}
    \label{fig:octagon_examples_radial}
\end{figure}

\section{Discussion}
\label{sec:concl}
The Euler characteristic is perhaps the most ubiquitous topological invariant and is certainly one of the easiest to compute efficiently. Despite its simple nature, it often captures many interesting features of dataset, particularly when randomness is involved, where it is one of the few cases where closed-form expressions exist. While the Euler characteristic curve has been used extensively in applications for some time now, multi-parameter analogues have not been explored. Here we began this exploration, as the general interest in multi-parameter persistence theory shows, the need to study higher-dimensional parameter spaces is becoming increasingly. 

Here we showed that considering these Euler surfaces are useful as features for classification tasks, but also for the kind of qualitative analysis that is one of the benefits of TDA. The identification of ranges of parameter values which differentiate models is a very interesting observation for data analysis.  These promising results also naturally lead to interesting mathematical questions which have yet to be explored. 

\begin{enumerate}
    \item Can one do interesting analysis on the Euler surfaces. While we have shown the particular results, will a clustering algorithm run on Figure ~\ref{octagon_landscape_regions}, always yield sensible and more importantly meaningful results? Furthermore, are there interesting interpretations of multiple regions of the Euler surface providing good differentiation? 
    \item Our current approach to the Euler surfaces are point-wise, however, can a more functional approach yield additional insights. In particular, are there basis functions which accurately capture the ``shape" of the curves? In future work, we plan to examine the result of techniques such as functional PCA. 
    \item As mentioned, the Euler surfaces can be thought of as something of a sample or a subset from the Euler characteristic transform (ECT).  An active area of research is to understand how to evaluate and characterize the stability of the ECT. While Euler surfaces are not a solution to this -- the identification of differentiating parameter ranges may be one approach to this problem. A completely open question remains how to interpret this and understand this phenomenon within the broader theory of constructible sheaf transforms.
    \item We often would like to encode certain invariances into an analysis, e.g. rotational invariance. In the case of Euler surcase this can be done by the choice of filtering function, as is the case for distance from the center. A natural extension of this, is can these invariances be extended to the ECT, so that they remain invertible up to an invariance. This seems likely but remains an open question.
\end{enumerate}

\bibliographystyle{unsrt}  
\bibliography{euler}  

\end{document}